\documentclass[journal]{IEEEtran}
\usepackage{amsmath, amssymb, amsthm}

\usepackage{cleveref}
\usepackage{graphicx} 
\usepackage{tikz}
\usetikzlibrary{positioning}
\usepackage{pgfplots}
\usepackage{tikz-3dplot}
\usetikzlibrary{arrows.meta}
\usetikzlibrary{cd,calc}
\usetikzlibrary{decorations.pathmorphing}
\usetikzlibrary{3d,calc}
\tikzset{>=latex}
\graphicspath{{./Figs/}}
\usepgfplotslibrary{colorbrewer}
\pgfplotsset{compat=1.18}
\usepackage{booktabs}

\definecolor{myblue}{RGB}{87, 163, 213}
\ifCLASSOPTIONcompsoc
    \usepackage[caption=false, font=normalsize, labelfont=sf, textfont=sf]{subfig}
\else
\usepackage[caption=false, font=footnotesize]{subfig}
\fi

\newcommand{\E}{\mathbb{E}}
\renewcommand{\H}{\mathsf{H}}
\newcommand{\Q}{\mathsf{Q}}
\newcommand{\X}{\mathsf{X}}
\newcommand{\G}{\mathsf{G}}
\newcommand{\T}{\mathsf{T}}
\newcommand{\pair}[1]{\ensuremath{\left\langle #1 \right\rangle}}
\renewcommand{\Re}{\ensuremath{\mathbb{R}}}

\newcommand{\deriv}[2]{\ensuremath{\frac{\partial #1}{\partial #2}}}

\newcommand{\parenth}[1]{\ensuremath{\left( #1 \right)}}
\renewcommand{\div}{\mathrm{div}}
\newcommand{\metric}{\ensuremath{g}}
\newcommand{\Sph}{\mathsf{S}}

\newtheorem{definition}{Definition}
\newtheorem{theorem}{Theorem}
\newtheorem{remark}{Remark}
\newtheorem{assumption}{Assumption}
\newtheorem{example}{Example}
\Crefname{assumption}{Assumption}{Assumptions}

\title{Transfer Operators for Stochastic Hybrid Systems on Manifolds with Guard-Induced Resets}

\author{Tejaswi K. C., William A. Clark, and Taeyoung Lee
	\thanks{Tejaswi K.C., Taeyoung Lee, Mechanical and Aerospace Engineering, George Washington University, Washington, DC 20052, {\tt \{kctejaswi99,tylee\}@gwu.edu}}%
    \thanks{William A. Clark, Mathematics, Ohio University, Athens, OH 45701 {\tt clarkw3@ohio.edu}}
	\thanks{\textsuperscript{\footnotesize\ensuremath{*}}This research has been supported in part by AFOSR MURI FA9550-23-1-0400, and ONR N00014-23-1-2850.}
}

\begin{document}
	\allowdisplaybreaks
	\maketitle \thispagestyle{empty} \pagestyle{empty}

\begin{abstract}
	This paper develops a transfer operator framework for stochastic hybrid systems with guard-induced resets, encompassing both the Koopman and Frobenius--Perron operators.  
	Exploiting their duality, we derive a unified formulation in which observables and probability densities evolve under adjoint generators corresponding to the backward and forward Kolmogorov equations.  
	The formulation is developed in a global and intrinsic manner on differentiable manifolds, ensuring consistency with the underlying geometric structure of the state space.  
	In addition, we propose a finite volume computational scheme on manifolds that preserves total probability mass while accurately capturing fluxes across guards and reset-induced transfers.
	The proposed framework provides a unified and geometrically consistent approach to uncertainty propagation in stochastic hybrid systems, bridging continuous stochastic dynamics and hybrid transitions within a transfer operator perspective.
\end{abstract}

\section{Introduction}

Hybrid dynamical systems provide a natural mathematical framework for modeling systems that exhibit both continuous-time evolution and event-driven, instantaneous changes~\cite{goebel2012hybrid}.  
This framework has been extended to stochastic hybrid systems~\cite{hespanha2004shs,hu2000towards,teel2014survey} to incorporate randomness and uncertainty in both the continuous and discrete dynamics.  
Such an extension is essential in realistic settings, where uncertainties are unavoidable due to stochastic disturbances, imperfect sensing, and environmental variability, and it has been applied across a wide range of domains, including communication networks~\cite{hespanha2004shs}, biological modeling~\cite{li2017shs-bio}, and neuroscience~\cite{bressloff2018cell-neuro}.

In stochastic dynamical systems, propagating uncertainty at the distributional level yields a deterministic and global description of system behavior, capturing the evolution of all possible realizations simultaneously and revealing intrinsic transport and transfer mechanisms across modes.  
This perspective is naturally formulated in terms of transfer operators, namely the Koopman and Frobenius--Perron operators, which are dual to each other: the Koopman operator governs the evolution of observables, while the Frobenius--Perron operator governs the evolution of probability densities~\cite{mezic2020spectrum,ikeda2022koopman}.  
At the infinitesimal level, their generators are adjoint operators, corresponding to the backward Kolmogorov equation for observables and the forward Kolmogorov (Fokker--Planck) equation for densities~\cite{risken1996fpe,pavliotis2014spa}.

Existing approaches to uncertainty propagation in hybrid systems can be broadly categorized as follows.  
For deterministic hybrid systems, transfer operator frameworks have been developed, and techniques such as modified Lie--Poisson reduction have been applied to reduce the complexity of mechanical impact systems, enabling efficient analysis of long-term asymptotic behavior~\cite{oprea2024study}.  
In addition, data-driven approaches based on Koopman operator theory have been proposed to learn representations of hybrid dynamics~\cite{bakker2020learningkoopmanrepresentationshybrid}.  
For stochastic hybrid systems with spontaneous jumps, uncertainty propagation has been studied in settings where transitions are triggered by Poisson processes, leading to coupled evolution equations that account for both diffusion and random switching mechanisms~\cite{hespanha2005model}.  
Such formulations have been utilized for Bayesian estimation~\cite{WanLeePACC20b} and extended to rigid body dynamics on Lie groups, where intrinsic computational schemes are developed via non-commutative harmonic analysis~\cite{WanLeeSJADS22}.  
Furthermore, it has been shown that multiple formulations of stochastic hybrid systems can arise depending on how randomness is introduced, as illustrated through one-dimensional examples~\cite{CClaLeePISNCS25}.

However, these formulations do not fully address the propagation of uncertainty in stochastic hybrid systems where discrete transitions are deterministically triggered by guard conditions and reset maps.  
In such systems, probability mass evolves continuously within each mode while undergoing discontinuous transfer across guards, resulting in a coupled evolution that is not captured by classical formulations developed for either purely continuous systems or spontaneously switching systems.  
This interplay between stochastic continuous flows and deterministic jumps arises naturally in many applications.  
For example, in the bouncing ball dynamics—one of the most fundamental examples in hybrid systems—the motion of the ball in the air is subject to stochastic perturbations, such as aerodynamic drag, while the impact event is deterministically triggered when the state reaches a guard set (the ground).

The objective of this paper is to develop a consistent transfer operator framework for stochastic hybrid systems with guard-induced resets.  
Specifically, we derive a density evolution equation that couples within-mode Fokker--Planck dynamics with explicit inter-mode transfer terms induced by reset maps, while ensuring conservation of total probability mass.  
This formulation provides a systematic and physically meaningful characterization of uncertainty propagation across hybrid modes.

To achieve this, we build on the duality between the Koopman and Frobenius--Perron operators and focus on the evolution of probability densities.  
The central idea is to represent the reset-induced transfer of probability explicitly as the image of probability flux across the guard under the reset map.  

This leads to a unified and interpretable description in which continuous transport and discrete transitions are incorporated within a single density evolution framework.  
Moreover, the proposed formulation is developed in a global and intrinsic manner on differentiable manifolds, avoiding reliance on local coordinates and ensuring consistency with the underlying geometric structure of the state space.

In short, the main contributions of this paper are summarized as follows:
\begin{itemize}
\item We develop a Frobenius--Perron formulation for stochastic hybrid systems with guard-induced resets, in which the evolution of probability densities is governed by coupled Fokker--Planck dynamics within each mode and inter-mode transfer mechanisms.

\item We establish a unified density evolution framework that captures both continuous transport and discontinuous transitions while ensuring conservation of total probability mass across modes.

\item We formulate the proposed framework in a global and intrinsic manner on differentiable manifolds, enabling a coordinate-free treatment that is consistent with the underlying geometric structure of the state space.

\item We develop a finite volume computational scheme on differentiable manifolds for uncertainty propagation, which preserves key conservation properties, including total probability mass, while accurately capturing fluxes across guard sets and reset-induced transfers between modes.
\end{itemize}

This paper is organized as follows.  
Stochastic hybrid systems and transfer operators are formulated in \Cref{sec:SHS}.  
Stochastic hybrid systems with a single mode are then presented in \Cref{sec:SHS1}, along with a one-dimensional example illustrating multiple types of hybrid behavior.  
Next, in \Cref{sec:multi_mode}, stochastic hybrid systems with multiple modes are analyzed, with an example of hybrid dynamics on a torus.  
Finally, conclusions are provided in \Cref{sec:conclusions}.

\section{Stochastic Hybrid Systems}\label{sec:SHS}

In this section, we formulate general stochastic hybrid systems (GSHS), which are uncertain dynamical systems that may undergo continuous-time evolution and discrete jumps. 
We then introduce the stochastic Koopman operator and the stochastic Frobenius--Perron operator that characterize the evolution of observables and densities.

\subsection{General Stochastic Hybrid Systems}

A general stochastic hybrid system (GSHS) is defined by the tuple $\mathcal{H} = (\H,\mathcal{F}, \mu, X, h, \G, \Phi)$
described as follows.

\begin{itemize}

	\item \textbf{Hybrid state space.}
		Let $\Q=\{1,2,\ldots,N_Q\}$ be the set of discrete modes. 
		For each $q\in\Q$, let $\X_q$ denote an $n_q$-dimensional Riemannian manifold with boundary.
		The hybrid state space is the disjoint union
		\begin{align}
			\H=\bigsqcup_{q\in\Q}\X_q = \bigcup_{q\in\Q}\{(x,q)\mid x\in\X_q\},
		\end{align}
		where $(x,q)$ denotes the hybrid state consisting of the continuous state $x$ and the discrete mode $q$.

	\item \textbf{$\sigma$-algebra.}
		The $\sigma$-algebra on $\H$ is defined by
		\begin{align}
			\mathcal{F}
			=
			\Bigl\{\bigsqcup_{q\in\Q} A_q \,\Big|\, A_q\in\mathcal{B}(\X_q)\Bigr\},
			\label{eqn:F_hyb}
		\end{align}
		where $\mathcal{B}(\X_q)$ denotes the Borel $\sigma$-algebra on $\X_q$.

	\item \textbf{Reference measure.}
		The measure $\mu:\mathcal{F}\to\mathbb{R}_{\ge0}$ is defined by
		\begin{align}
			\mu(A)=\sum_{q\in\Q}\mu_q(A\cap\X_q),
			\label{eqn:mu_hyb}
		\end{align}
		where $\mu_q$ is the Riemannian volume measure on $\X_q$.

	\item \textbf{Continuous stochastic dynamics.}
		In each mode $q$, the continuous state evolves according to the It\^{o} stochastic differential equation
		\begin{equation}
			dx = X(x,q)\,dt + \sum_{i=1}^{m} h_i(x,q) dW_i,
			\label{eqn:ContSDE}
		\end{equation}
		where $X(\cdot,q)\in\mathfrak{X}(\X_q)$ is the drift vector field, 
		$h_i(\cdot,q)\in\mathfrak{X}(\X_q)$ are diffusion vector fields, and $W_i$ are independent Wiener processes.

	\item \textbf{Guard set.}
		The guard set is $\G=\bigsqcup_{q\in\Q}\G_q$, where $\G_q\subset\partial\X_q$.
		A discrete transition is triggered when $(x,q)$ reaches $\G_q$.

	\item \textbf{Reset map.}
		The post-jump state is determined by a measurable reset map $\Phi:\G\to\H$, defined by
		\[
			(x^+,q^+) = \Phi(x^-,q^-).
		\]

\end{itemize}

The integral of an integrable function $f:\H\to\mathbb{R}$ is defined by 
\begin{align}
    \int_{\H} f(x,q)\,d\mu(x,q)
    =
    \sum_{q\in\Q}\int_{\X_q} f(x,q)\,d\mu_q(x).
\end{align}

Let $L^p(\H)$ denote the space of measurable functions $f:\H\to\mathbb{R}$ such that $|f|^p$ is integrable.
For $f\in L^\infty(\H)$ and $g\in L^1(\H)$, define the dual pairing
\begin{align}
    \pair{g,f}
    =
    \sum_{q\in\Q}\int_{\X_q} g(x,q)\,f(x,q)\,d\mu_q(x).
    \label{eqn:pair}
\end{align}

This formulation describes stochastic hybrid systems with continuous diffusion processes on Riemannian manifolds coupled with deterministic, guard-triggered discrete transitions.
It complements prior work on deterministic hybrid systems~\cite{oprea2024study} and stochastic hybrid systems with Poisson-driven jumps~\cite{hespanha2005stochastic}.

\subsection{Stochastic Koopman and Frobenius--Perron Operators}

The stochastic Koopman and Frobenius--Perron operators provide a unified framework for analyzing stochastic dynamical systems by connecting probabilistic, dynamical, and operator-theoretic viewpoints.

Let $z_t=(x_t,q_t)$ denote the hybrid stochastic process evolving on $\H$.
For deterministic systems, the Koopman operator $\mathcal{K}_t:L^\infty(\H)\to L^\infty(\H)$ is defined by $\mathcal{K}_t f(z)=f(z_t)$ for $z_0=z$.
Its stochastic counterpart is obtained by taking conditional expectation~\cite{wanner2022robust}.

\begin{definition}\label{def:SK}
    The stochastic Koopman operator $\mathcal{K}_t:L^\infty(\H)\to L^\infty(\H)$ for $t>0$ is defined by
    \begin{align}
        \mathcal{K}_t f(z) = \E[f(z_t)\mid z_0=z],
        \label{eqn:K0}
    \end{align}
    for $f\in L^\infty(\H)$.
    If the limit
    \begin{align}
        \mathcal{A}f(z) = \lim_{t\downarrow0}\frac{\mathcal{K}_t f(z)-f(z)}{t}
        \label{eqn:AA}
    \end{align}
    exists, then $f$ belongs to the domain of $\mathcal{A}$, and $\mathcal{A}$ is called the infinitesimal generator.
\end{definition}

The family $\{\mathcal{K}_t\}_{t\ge0}$ forms a semigroup,
\begin{align}
    \mathcal{K}_{t+s}=\mathcal{K}_t\circ\mathcal{K}_s,
    \label{eqn:semigroup}
\end{align}
and the generator satisfies $\mathcal{A}\mathcal{K}_t=\mathcal{K}_t\mathcal{A}$.
Let $u(t,z)=\mathcal{K}_t f(z)$.
If $f$ lies in the domain of $\mathcal{A}$, then
\begin{align}
    \frac{\partial u(t,z)}{\partial t} = \mathcal{A}u(t,z).
    \label{eqn:u_dot_Au}
\end{align}

The Frobenius--Perron operator is defined as the dual of the Koopman operator.
This definition extends naturally to stochastic and hybrid systems, where a direct flow-based interpretation is not available.

\begin{definition}
    The stochastic Frobenius--Perron operator $\mathcal{P}_t:L^1(\H)\to L^1(\H)$ is defined by
    \begin{align}
        \pair{g,\mathcal{K}_t f} = \pair{\mathcal{P}_t g,f},
        \label{eqn:adj_K_FP}
    \end{align}
    for all $f\in L^\infty(\H)$ and $g\in L^1(\H)$.
\end{definition}

From the duality \eqref{eqn:adj_K_FP} and the semigroup property \eqref{eqn:semigroup}, the family $\{\mathcal{P}_t\}_{t\ge0}$ also forms a semigroup,
\begin{align}
    \mathcal{P}_{t+s}=\mathcal{P}_t\circ\mathcal{P}_s,
\end{align}
whose generator is the adjoint $\mathcal{A}^*$ of $\mathcal{A}$.
Let $v(t,z)=\mathcal{P}_t g(z)$.
Then
\begin{align}
    \frac{\partial v(t,z)}{\partial t} = \mathcal{A}^*v(t,z).
    \label{eqn:v_dot_A*v}
\end{align}

A key role of the Frobenius--Perron operator is to propagate probability densities.
Let $p:\mathbb{R}_{\ge0}\times\H\to\mathbb{R}_{\ge0}$ be the density for the hybrid state satisfying
\[
    \mathbb{P}[z_t\in A]=\int_A p(t,z)\,d\mu(z),
\]
for $A\in\mathcal{F}$, where the left-hand side denotes the probability that $z_t$ belongs to $A$. 
Then for any test function $f:\H\to\mathbb{R}$, from the law of total expectation~\cite{weiss2006course}, we have
\begin{align*}
    \E[f(z_t)] &=\E[\E[f(z_t)\mid z_0=z]],
\end{align*}
where the outer mean is taken with respect to the initial state $z_0$.
More explicitly, using \eqref{eqn:K0}, 
\begin{align*}
	\E[f(z_t)] &=\int_{\H} p(0,z)\,\mathcal{K}_t f(z)\,d\mu(z)  = \pair{p(0,z), \mathcal{K}_t f(z)} \\
               &= \pair{ \mathcal{P}_t p(0,z), f(z)},
\end{align*}
where the last equality is from the duality given by \eqref{eqn:adj_K_FP}.
Since this holds for any arbitrary $f$, it implies that the density evolves as
\begin{align}
    p(t,z)=\mathcal{P}_t p(0,z).
    \label{eqn:FP0}
\end{align}
Thus \eqref{eqn:v_dot_A*v} corresponds to the Fokker--Planck equation when $v(0,z)$ is chosen as the initial density.

In the following sections, we derive explicit expressions for the generators $\mathcal A$ and $\mathcal A^*$ for stochastic hybrid systems.
We first consider the case of a single continuous mode to illustrate the generator structure and boundary identities, and then extend the results to stochastic hybrid systems with multiple modes and reset transitions.
The presence of guard-triggered resets introduces additional structure that is not captured by classical generator expressions, and requires a careful characterization of how probability is transferred across modes.

\section{Stochastic Hybrid Systems with a Single Mode}\label{sec:SHS1}

In this section, we consider the special case of a stochastic hybrid system with a single discrete mode, i.e., $\Q=\{1\}$. 
The hybrid state space reduces to an $n$-dimensional Riemannian manifold $\X$ with boundary and metric $\metric$.
The objective is to construct the stochastic Koopman and Frobenius--Perron operators for this system. 
These results will serve as the building block for the general multi-mode case developed in Section~\ref{sec:multi_mode}.
Throughout this section, the mode index is suppressed for brevity. 

\subsection{Problem Formulation}

\begin{assumption}[Continuous flow]\label{assump:flow}
    The stochastic differential equation for the continuous flow \eqref{eqn:ContSDE} can be written as
    \begin{align}
        dx = X(x)\,dt + \sum_{i=1}^{n}\sigma E_i(x)\,dW_i,
        \label{eqn:SDE1}
    \end{align}
    where $X\in\mathfrak{X}(\X)$ is the drift vector field, $\sigma>0$ is a constant diffusion strength, and $\{E_1(x),\ldots,E_n(x)\}$ is an orthonormal frame of $\X$.
\end{assumption}

In other words, for each $x\in\X$, the set $\{E_1(x),\ldots,E_n(x)\}$ forms a basis of the tangent space $T_x\X$ satisfying $\metric(E_i(x),E_j(x))=\delta_{ij}$, and the diffusion field is applied along these basis directions with equal strength $\sigma$. 
This construction is motivated by Brownian motion on a Riemannian manifold~\cite{lee2025brownian}, which yields an isotropic diffusion process. 
Although the results can be extended to anisotropic diffusions, we focus on the isotropic case to highlight the interaction between continuous stochastic flows and discrete jumps explicitly, without introducing additional complexities.

Next, we specify the boundary behavior and reset structure.
The reflecting boundary and absorbing boundary are defined as follows. 
\begin{definition}[Reflecting and Absorbing Boundaries]\label{def:boundary_single}
    Let $u(t,x), v(t,x):\Re\times\X\rightarrow\Re$ be the Koopman observable and the density function, respectively.

    \begin{itemize}
        \item A point $x\in\partial \X$ is said to be \underline{reflecting} if
        \begin{align}
            \nabla u(t,x)\cdot N(x) = 0, \qquad Y(t,x)\cdot N(x) = 0, \label{eqn:BC_uv_reflecting_single}
        \end{align}
        where $Y(t,\cdot)\in\mathfrak{X}(\X)$ is defined by
        \begin{align}
            Y(t,x) = v(t,x) X(x) - H \nabla v(t,x),
        \end{align}
        with $H=\frac{1}{2}\sigma^2>0$ and $N(x)$ is the outward-pointing unit normal vector at $x\in\partial \X$.

        \item A point $x\in\partial \X$ is said to be \underline{absorbing} if
        \begin{align}
            u(t,x) = 0, \qquad v(t,x) = 0. \label{eqn:BC_uv_absorbing_single}
        \end{align}
    \end{itemize}
\end{definition}
The vector field $Y=vX - H\nabla v$ represents the probability flux associated with the density $v$. 
Thus, the condition $Y\cdot N=0$ at a reflecting boundary indicates that no probability mass crosses the boundary, and all trajectories are confined within $\X$. 
The accompanying Neumann condition $\nabla u\cdot N=0$ ensures that the Koopman observable is consistent with this confinement. 
In contrast, at an absorbing boundary, the conditions $v=0$ and $u=0$ imply that trajectories reaching the boundary are removed from the system and contribute no future evolution, leading to a loss of total probability mass. 

\begin{assumption}[Boundary Decomposition and Reset Structure]\label{assump:jump}
    The boundary of $\X$ is decomposed into a disjoint union of the guard $\G$ and the physical boundary $\Gamma$ as
    \[
        \partial\X=\G\cup\Gamma, \qquad \G\cap\Gamma=\emptyset.
    \]
    Furthermore, the following properties are satisfied.

    \begin{itemize}

        \item \textbf{Reset on the guard.}
            Upon intersecting with the guard $\G$, the state resets according to
            \begin{align}
                x^+ = \Phi(x), \qquad x\in\G,
                \label{eqn:reset1}
            \end{align}
            where the reset image lies on the physical boundary,
            \[
                \Phi(\G)\subset\Gamma .
            \]

        \item \textbf{Guard trace condition.}
            Trajectories reaching the guard are instantaneously transferred by the reset map and therefore do not remain on $\G$. 
            Accordingly, the density satisfies
            \begin{align}
                v(t,x)=0, \qquad x\in\G .
                \label{eqn:BCv_guard_single}
            \end{align}

        \item \textbf{Guard--flux compatibility.}
            Probability mass leaving the domain through $\G$ is transferred by the reset map to its image $\Phi(\G)$ without accumulation or loss. 
            In the weak formulation, the outgoing flux across $\G$ is exactly balanced by the incoming flux on $\Phi(\G)$, so that the net contribution of each guard/reset-image pair vanishes.

        \item \textbf{Physical boundary on $\Gamma$.}
            The set $\Gamma$ is partitioned into reflecting and absorbing parts,
            \[
                \Gamma=\Gamma^{\rm ref}\cup\Gamma^{\rm abs},
                \qquad
                \Gamma^{\rm ref}\cap\Gamma^{\rm abs}=\emptyset .
            \]
            For $x\in\Gamma^{\rm ref}$ and $x\in\Gamma^{\rm abs}$, the boundary conditions in \Cref{def:boundary_single} are imposed.

            In particular, points in $\Phi(\G)$ may belong to $\Gamma^{\rm ref}$ or $\Gamma^{\rm abs}$; in such cases, the local boundary condition, such as \eqref{eqn:BC_uv_reflecting_single}, is combined with the incoming flux induced by the reset.
    \end{itemize}

\end{assumption}
In summary, the boundary of $\X$ consists of two distinct components.
The guard $\G$ acts as a transport boundary, through which probability mass is not lost but instead relocated by the reset map $\Phi$.
In contrast, $\Gamma$ represents the physical boundary, where reflecting parts confine the system and absorbing parts remove probability mass.
Thus, reflecting boundaries enforce zero flux, absorbing boundaries allow loss, and guards redirect probability mass through deterministic resets.
This distinction between transport through the guard and physical boundary behavior plays a central role in the operator-based formulation developed in the following sections.

\subsection{Stochastic Koopman/Frobenius-Perron Operator}

Next, we develop the stochastic Koopman operator and the Frobenius-Perron operator as follows. 

\begin{theorem}\label{thm:K2}
    Consider a hybrid system defined by \eqref{eqn:SDE1}--\eqref{eqn:reset1}.
    Let $f\in L^\infty(\X)$ and define $u(t,x)=\mathcal K_t f(x)$. Then $u$ satisfies
    \begin{align}
        \dfrac{\partial u(t,x)}{\partial t} & = \mathcal{A} u(t,x)\quad \mbox{for } x\in\X\setminus\G,\label{eqn:dotu2} \\
        u(t, x) & = u(t, \Phi(x)) \quad \mbox{for } x\in\G,\label{eqn:BCu2}
    \end{align}
    with the initial condition $u(0,x) = f(x)$, where the generator $\mathcal{A}$ is given on its domain by 
    \begin{align}
        \mathcal A u(t,x) &= \mathcal{L}_X u(t, x) + H \Delta u(t,x), \label{eqn:A2}
    \end{align}
    where $\Delta(\cdot) = \mathrm{div}_\mu(\nabla(\cdot))$ denotes the Laplace--Beltrami operator, and $H=\frac{1}{2}\sigma^2>0$.
\end{theorem}
\begin{proof}
    According to \cite[Theorem 4]{lee2025brownian}, the generator of the diffusion part of  \eqref{eqn:SDE1} yields half of the Laplace-Beltrami operator, scaled by $\sigma^2$.
    Thus, including the advection part of \eqref{eqn:SDE1}, the complete generator is
    \begin{align*}
        \mathcal{A} u = \mathcal{L}_X u + H \Delta u,
    \end{align*}
    where $\mathcal{L}_X f = X[f]= df(X)$ denotes the Lie-derivative. 
    This yields \eqref{eqn:A2}. 

    Next, we show the boundary condition on $\G$. 
    Let $x\in\G$.
    By the definition of the hybrid system, it is instantaneously reset according to $x^+=\Phi(x)$. 
    Thus, starting from $x\in\G$, the process immediately transitions to $\Phi(x)$ and continues its evolution from that point.
    Therefore, for any $t>0$,
    \begin{align*}
        u(t,x) & = \E[f(z_t)\mid z_0=x] = \E[f(z_t)\mid z_0=\Phi(x)]\\
               & = u(t,\Phi(x)),
    \end{align*}
    which establishes \eqref{eqn:BCu2}.
\end{proof}

This shows that the stochastic Koopman operator evolves according to the same generator as the underlying diffusion in the interior of the state space, while the forced jump at the guard induces the compatibility condition \eqref{eqn:BCu2} across the reset.
Next, the stochastic Frobenius--Perron operator is obtained by duality.

\begin{theorem}\label{thm:P2}
    Consider the stochastic hybrid system satisfying \Cref{assump:flow,assump:jump}. 
    Let $g\in L^1(\X)$ be an initial density, and let $ v(t,x)=\mathcal P_t g(x)$ denote the corresponding stochastic Frobenius--Perron evolution. Then
    \begin{align}
        \frac{\partial v(t,x)}{\partial t}
        =
        \mathcal A^*v(t,x),
        \qquad x\in\X\setminus\G,
        \label{eqn:dotv2}
    \end{align}
    with initial condition $v(0,x)=g(x)$. 
    The adjoint generator $\mathcal A^*$ is given on its domain by
    \begin{align}
        \mathcal A^*v(t,x)
        =
        -\mathcal L_X v(t,x)-v(t,x)\,\div_\mu(X)+H\Delta v(t,x),
        \label{eqn:A*2}
    \end{align}
    with $H=\frac12\sigma^2$. Equivalently, in divergence form,
    \begin{align}
        \mathcal A^*v(t,x)
        =
        -\div_\mu(Y(t,x)),
        \label{eqn:A*2_div}
    \end{align}
    where $Y(t,\cdot)\in\mathfrak{X}(\X)$ is
    \begin{align}
        Y(t,x)=v(t,x)X(x)-H\nabla v(t,x).
        \label{eqn:Y2}
    \end{align}

    Moreover, adjointness of the Koopman and Frobenius--Perron generators requires the boundary compatibility conditions
    \begin{gather}
        \int_{\partial\X} u(t,x)\,(Y(t,x)\cdot N(x))\,\nu = 0,
        \label{eqn:BC_P1}\\
        \int_{\partial\X} v(t,x)H\,(\nabla u(t,x)\cdot N(x))\,\nu = 0,
        \label{eqn:BC_P2}
    \end{gather}
    for every $u(t,x)$ in the domain of the Koopman generator, namely smooth in the interior and satisfying the boundary conditions induced by \Cref{assump:jump} and \eqref{eqn:BCu2}.

    Here $N$ is the outward-pointing unit normal vector field along $\partial\X$, and
    \[
        \nu=\iota^*(i_N\mu)
    \]
    is the induced $(n-1)$-form on $\partial\X$, where $\iota:\partial\X\hookrightarrow\X$ is the inclusion.
\end{theorem}

\begin{proof}
    Let $u(t,x)$ be the solution of \eqref{eqn:dotu2} and \eqref{eqn:BCu2} for the stochastic Koopman operator. 
	Using the expression \eqref{eqn:A2} for the generator,
    \begin{align}
        \pair{v,\mathcal Au}
        =
        \int_\X \Bigl(v\,\mathcal L_Xu + vH\,\div_\mu(\nabla u)\Bigr)\,\mu.
        \label{eqn:pair_P2}
    \end{align}

    Applying the integration-by-parts formula \eqref{eqn:integ_part} to the second term yields
    \begin{align*}
        \int_\X vH\,\div_\mu(\nabla u)\,\mu
        =
        \int_{\partial\X} vH\,(\nabla u\cdot N)\,\nu
        -
        \int_\X \langle \nabla(vH),\nabla u\rangle\,\mu.
    \end{align*}
    Since $H=\frac12\sigma^2$ is constant, $\nabla(vH)=H\nabla v$, and therefore
    \begin{align*}
        \pair{v,\mathcal Au} &= \int_\X \Bigl(v\,\mathcal L_Xu - H\langle \nabla v,\nabla u\rangle\Bigr)\,\mu
        + \int_{\partial\X} vH\,(\nabla u\cdot N)\,\nu.
    \end{align*}
    Using $ v\,\mathcal L_Xu=\langle vX,\nabla u\rangle$, 
    we obtain
    \begin{align*}
        \pair{v,\mathcal Au} &= \int_\X \langle vX-H\nabla v,\nabla u\rangle\,\mu + \int_{\partial\X} vH\,(\nabla u\cdot N)\,\nu \\
                             &= \int_\X \langle Y,\nabla u\rangle\,\mu + \int_{\partial\X} vH\,(\nabla u\cdot N)\,\nu,
    \end{align*}
    where $Y=vX-H\nabla v$.

    Applying \eqref{eqn:integ_part} again, the first term is given by
    \begin{align}
        \int_\X \langle Y,\nabla u\rangle\,\mu
        =
        \int_{\partial\X} u\,(Y\cdot N)\,\nu
        -
        \int_\X u\,\div_\mu(Y)\,\mu.
        \label{eqn:thm_tmp0}
    \end{align}
    Hence
    \begin{align}
        \pair{v,\mathcal Au} & = \int_{\partial\X} u\,(Y\cdot N)\,\nu + \int_{\partial\X} vH\,(\nabla u\cdot N)\,\nu\nonumber\\
                             & \quad - \int_\X u\,\div_\mu(Y)\,\mu.
                             \label{eqn:thm_tmp0_revised}
    \end{align}

	To evaluate the boundary terms in \eqref{eqn:thm_tmp0_revised}, we decompose
    \[
        \partial\X=\G\cup\Gamma^{\rm ref}\cup\Gamma^{\rm abs},
    \]
	by \Cref{assump:jump}.

    First, consider the second boundary term:
    \[
        \int_{\partial\X} vH\,(\nabla u\cdot N)\,\nu =
        \int_{\G\cup \Gamma^{\rm ref} \cup \Gamma^{\rm abs} } vH\,(\nabla u\cdot N)\,\nu.
    \]
    On the guard $\G$, the trace condition \eqref{eqn:BCv_guard_single} gives $v=0$.
    On the reflecting boundary $\Gamma^{\rm ref}$, \eqref{eqn:BC_uv_reflecting_single} gives $\nabla u\cdot N=0$.
    Further, on the absorbing boundary $\Gamma^{\rm abs}$, \eqref{eqn:BC_uv_absorbing_single} gives $v=0$.
    Hence the integrand vanishes on each part of the boundary, and this yields \eqref{eqn:BC_P2}.

    Next, consider the first boundary term $\int_{\partial\X} u\,(Y\cdot N)\,\nu$.
    On the reflecting boundary $\Gamma^{\rm ref}$, \eqref{eqn:BC_uv_reflecting_single} gives $ Y\cdot N=0$, and
    on the absorbing boundary $\Gamma^{\rm abs}$, \eqref{eqn:BC_uv_absorbing_single} yields $u=0$.
	Thus the contribution from the physical boundary vanishes. 
	On the guard $\G$, the flux need not vanish pointwise. 
	However, by the reset mechanism, the mass leaving through $\G$ is exactly redistributed to $\Phi(\G)$, so that the net contribution of each guard/reset pair cancels in the weak formulation, thereby resulting in \eqref{eqn:BC_P1}.

	Substituting these into \eqref{eqn:thm_tmp0_revised}, we obtain
    \[
        \pair{v,\mathcal Au}
        =
        -\int_\X u\,\div_\mu(Y)\,\mu,
    \]
    which yields the duality with the fomulational of the adjoint in \eqref{eqn:A*2_div}.
    Using the product rule for divergence in \eqref{eqn:div_prod}, \eqref{eqn:A*2_div} is rearranged into \eqref{eqn:A*2}.
\end{proof}
This result shows that the effect of the reset appears entirely through boundary terms, while the interior evolution is governed by the same differential operator as the underlying diffusion.
The resulting boundary conditions for the Frobenius--Perron operator, given in \eqref{eqn:BC_P1} and \eqref{eqn:BC_P2}, are expressed in weak form, i.e., 
the boundary conditions are enforced through integral identities rather than pointwise constraints.
For specific boundary types and reset maps, these can be reduced to explicit pointwise conditions, as discussed in \Cref{sec:EX1D}. 

Next, we present an interpretation regarding the conservation of the total density and a specialization for deterministic flows. 
\begin{remark}[Flux]
    The adjoint equation \eqref{eqn:dotv2} can be written as
    \[
        \frac{\partial v(t,x)}{\partial t} = -\div_\mu(Y(t,x)).
    \]
    Integrating over $\X$ and applying the divergence theorem \eqref{eqn:div_thm} gives
    \begin{align}
        \frac{d}{dt}\int_\X v\,\mu = -\int_{\partial\X}(Y\cdot N)\,\nu. \label{eqn:total_density}
    \end{align}
    Thus, $Y\cdot N$ represents the outward probability flux through the boundary.
    In particular, if the admissible test function $u\equiv 1$ is used in \eqref{eqn:BC_P1}, the right-hand side of the above vanishes, indicating conservation of the total probability mass.
\end{remark}

\begin{remark}[Deterministic Flow]
	The results of \Cref{thm:K2,thm:P2} reduce to deterministic hybrid systems by setting $H\equiv0$.
    Further, assume that $\G\subset\partial\X$ and $\partial\X=\G\cup\Phi(\G)$, i.e., the boundary is composed of the guard and the image of the reset map.%
    \footnote{In contrast to the stochastic flow, the entire boundary need not be part of the guard. It is only required that the vector field points inward at $\Phi(\G)$.}
    Also assume that $X\cdot N>0$ for any $x\in\G$, and $X\cdot N<0$ for any $x\in\Phi(\G)$, implying that there is no grazing at the guard and that, after the jump, the flow returns to the interior of $\X$.
    Under these assumptions, \eqref{eqn:BC_P1} becomes
    \begin{align*}
        \int_{\partial\X} u\,(vX\cdot N)\,\nu &= \int_{\G} u\,(vX\cdot N)\,\nu\\
                                              &\quad +\int_\G (\Phi^*u)\,(\Phi^*(vX\cdot N))\,(\Phi^*\nu)=0,
    \end{align*}
    where the second equality follows from the change-of-variable formula.
    Since $u(x)=u(\Phi(x))$ on $\G$ by \eqref{eqn:BCu2}, and the above must hold for arbitrary $u$, we obtain
    \begin{align*}
        v(X\cdot N)\,\nu
        +
        (\Phi^*v)(\Phi^*(X\cdot N))(\Phi^*\nu)=0,
    \end{align*}
    for $x\in\G$.
    Hence
    \begin{align*}
        \Phi^*v
        =
        -\frac{(X\cdot N)\,\nu}{(\Phi^*(X\cdot N))(\Phi^*\nu)}\,v.
    \end{align*}
    Since $X\cdot N$ and $\Phi^*(X\cdot N)$ have opposite signs on $\G$, this becomes
    \begin{align*}
        \Phi^*v
        =
        \frac{|X\cdot N|\,\nu}{|\Phi^*(X\cdot N)|\,(\Phi^*\nu)}\,v
        =
        \frac{1}{\mathcal J_\mu^X(\Phi)}\,v,
    \end{align*}
    for $x\in\G$, which recovers the hybrid Jacobian $\mathcal J_\mu^X(\Phi)\in\Re$ introduced in~\cite{clark2023invariant}.
\end{remark}

\subsection{One-Dimensional Examples}\label{sec:EX1D}

The boundary compatibility conditions \eqref{eqn:BC_P1} and \eqref{eqn:BC_P2} are stated in weak integral form.
In this subsection, we illustrate how they specialize in several one-dimensional stochastic hybrid systems on $\X=[a,b]\subset\Re$, where $a<b$.
Our goal is not to derive boundary types from these weak conditions, but to verify that they are consistent with the reflecting, absorbing, and reset structures introduced in \Cref{assump:jump}.

For each example in this subsection, the reference measure is $\mu=dx$.
Hence the induced boundary form $\nu$ is a zero-form, namely a scalar function on $\partial\X$, with $\nu(a)=\nu(b)=1$.
The outward unit normal is $N(a)=-1$, $N(b)=1$.

From \eqref{eqn:A2} and \eqref{eqn:A*2}, the generator and its adjoint are
\begin{align}
    \mathcal A u &= X\frac{\partial u}{\partial x}+H\frac{\partial^2u}{\partial x^2}, \label{eqn:Au}\\
    \mathcal A^*v &= - X\frac{\partial v}{\partial x} - v\frac{\partial X}{\partial x} + H\frac{\partial^2v}{\partial x^2}. \label{eqn:Asv}
\end{align}
Moreover,
\[
	Y=vX-H\frac{\partial v}{\partial x}.
\]
Therefore, the weak boundary compatibility conditions \eqref{eqn:BC_P1}--\eqref{eqn:BC_P2} become
\begin{gather}
    -\,u(t,a)Y(t,a)+u(t,b)Y(t,b)=0, \label{eqn:BC_P1_ex}\\
    -\,v(t,a)H\frac{\partial u(t,a)}{\partial x} +v(t,b)H\frac{\partial u(t,b)}{\partial x}=0. \label{eqn:BC_P2_ex}
\end{gather}
These identities will be verified below for several specific boundary configurations.

\begin{example}[Reflecting Boundary]\label{ex:reflecting}
    Let $\X=[a,b]$. We take
    \[
        \G=\emptyset,\qquad \Gamma=\partial\X=\{a,b\},
    \]
    and prescribe reflecting boundary conditions on $\Gamma$.
    Thus, in the terminology of \Cref{assump:jump}, the entire boundary is reflecting:
    \[
        \Gamma^{\rm ref}=\partial\X,\qquad \Gamma^{\rm abs}=\emptyset.
    \]
    No reset map is involved in this example.

	\begin{figure}
		\centerline{
            \begin{tikzpicture}
                \footnotesize
                \draw[->] (0,0) -- (8,0) node[below] {$x$};
                \node (c) at (6.5,0) {};
                \node (a) at (1.5,0) {};
                \node (b) at ($(a)!0.5!(c)$) {};
                \draw[->] (a) to[out=50, in=130, loop, looseness=15] (a) {};
                \draw[->] (b) to[out=130, in=50, loop, looseness=15] (b) {};
                \draw (b |- 0, 0.15) -- ++(0,-0.3) node[below] {$b$};
                \draw (a |- 0,0.15) -- ++(0,-0.3) node[below] {$a$};
                \draw (c |- 0,0.15) -- ++(0,-0.3) node[below] {$c$};
                \draw[thick, myblue] (a) -- (b);
                \draw[->,myblue,thick] ($(a)!0.5!(b)$) -- ++(0.1,0);
            \end{tikzpicture}
		}
		\caption{\Cref{ex:reflecting}: the state space is $\X=[a,b]$ (shaded in blue), and both endpoints are reflecting. The vector field $X(x)=-\gamma(x-c)$ for $\gamma>0$ and $c>b$ points to the right toward $x=c$.}
        \label{fig:X_reflecting}
    \end{figure}
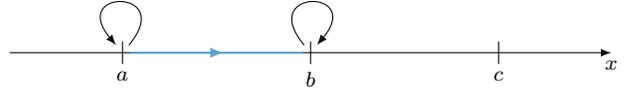

	\begin{figure}[h!]
        \centering
        \subfloat[Evolution of density $v$]{
        \includegraphics[width=0.48\columnwidth]{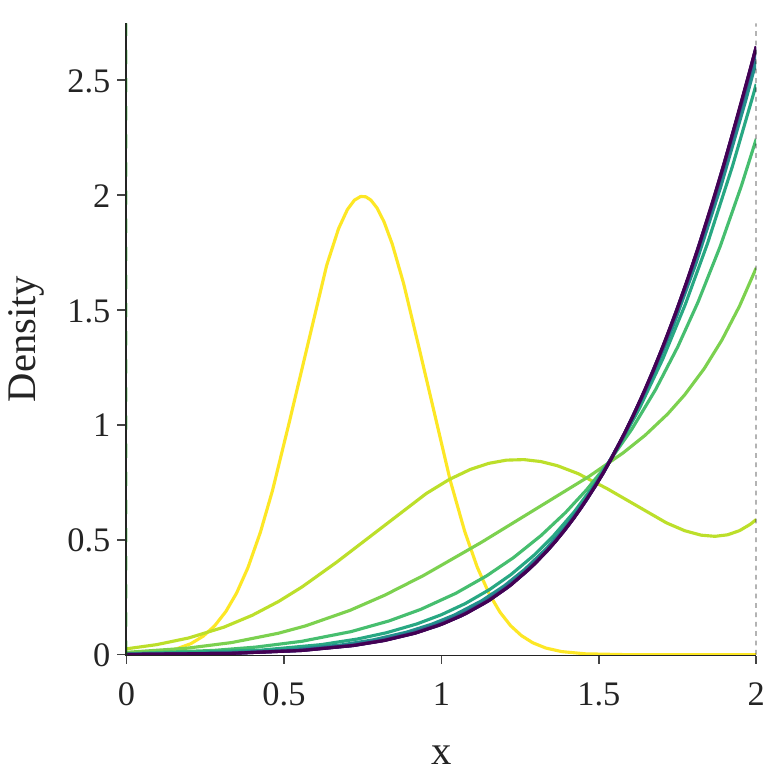}}
        \hfill
        \subfloat[Monte-Carlo simulation]{
        \includegraphics[width=0.48\columnwidth]{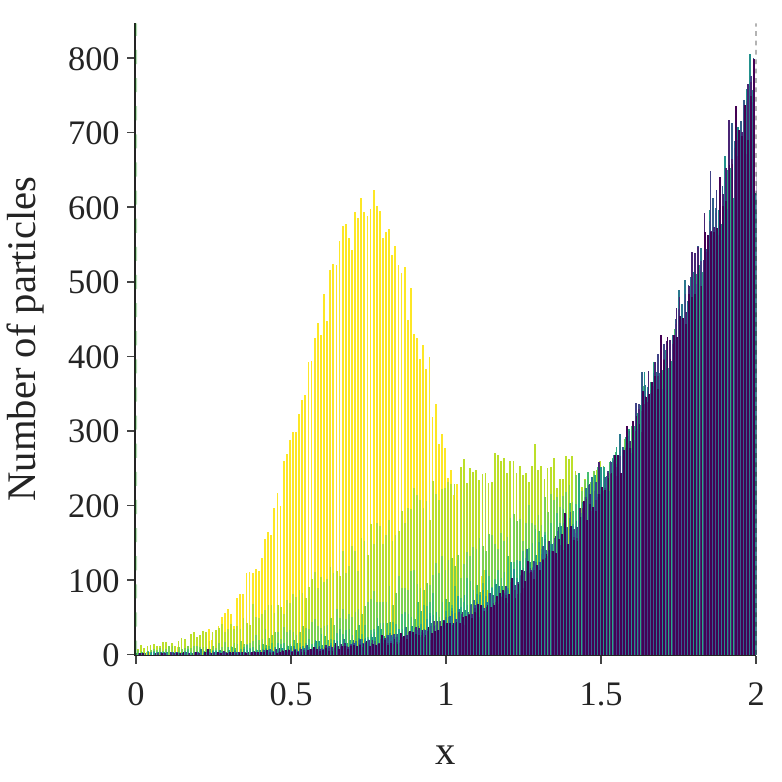}}
        \caption{\Cref{ex:reflecting}: numerical simulation of the Frobenius--Perron operator. (Throughout this section, time progression is represented by a gradual color change from yellow to blue.)}
        \label{fig:v_reflecting}
    \end{figure}
	\begin{figure}[h!]
        \centering
        \subfloat[Evolution of observable $u$]{
        \includegraphics[width=0.48\columnwidth]{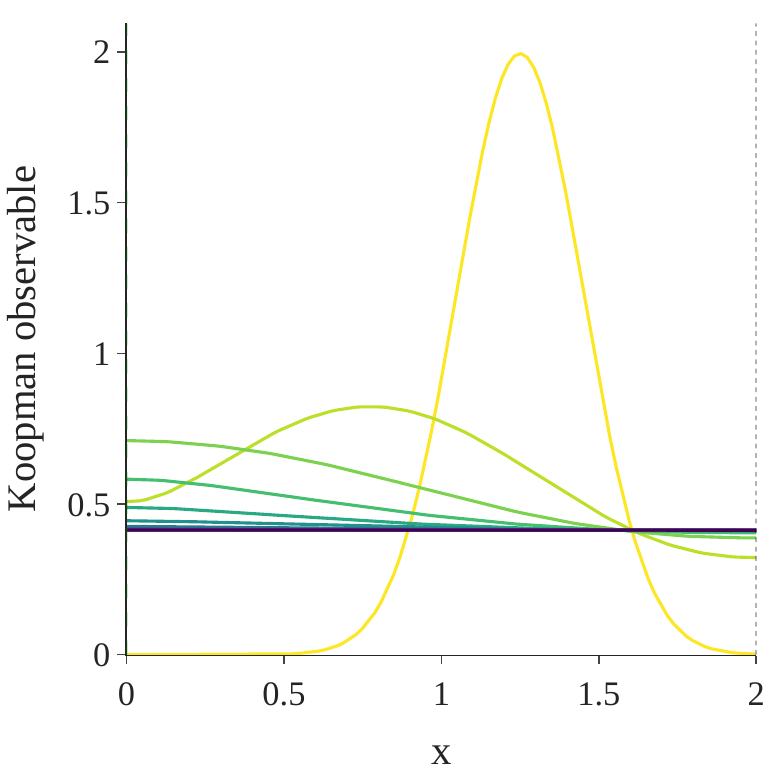}}
        \hfill
        \subfloat[Monte-Carlo simulation]{
        \includegraphics[width=0.48\columnwidth]{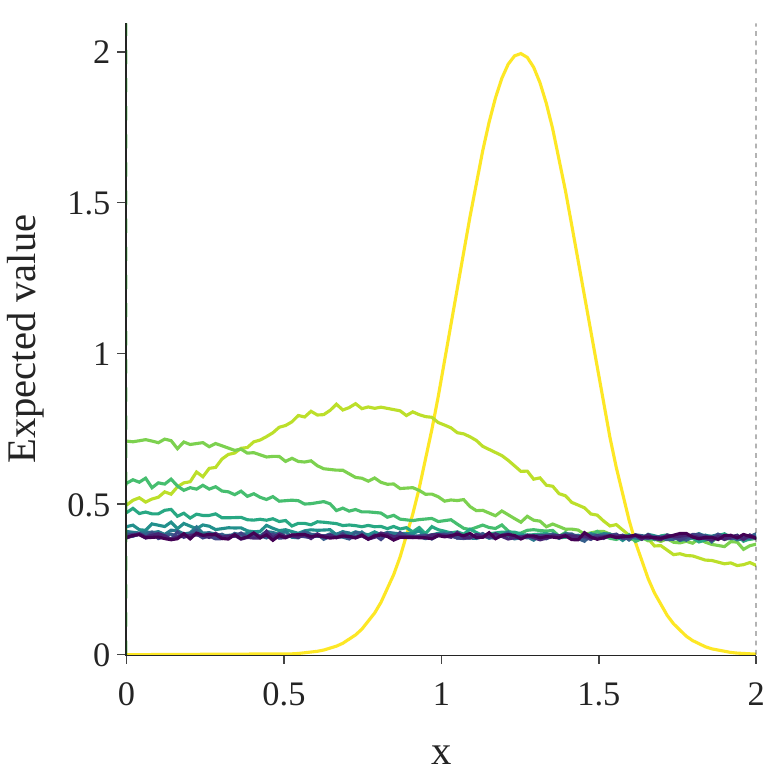}}
        \caption{\Cref{ex:reflecting}: numerical simulation of the Koopman operator.}
        \label{fig:u_reflecting}
    \end{figure}

    By definition of a reflecting boundary, the probability flux vanishes at the boundary:
    \[
        Y(t,a)=Y(t,b)=0,
    \]
    and the Koopman observable satisfies the Neumann condition
    \[
        \frac{\partial u(t,a)}{\partial x} = \frac{\partial u(t,b)}{\partial x} =0.
    \]
    Substituting these into \eqref{eqn:BC_P1_ex} and \eqref{eqn:BC_P2_ex}, we verify that the weak boundary compatibility conditions are satisfied.

    Consider a numerical example with vector field $X(x)=-\gamma(x-c)$ and noise coefficient $H=0.5$.
    We choose $a=0.0$, $b=2.0$, $c=3.0$, and $\gamma=1.0$.
    Here, the deterministic drift points toward the right, so trajectories tend to move toward the boundary at $x=b$.
    Because the boundaries are reflecting, trajectories remain in $[a,b]$ and are reflected upon reaching either endpoint.
    This is illustrated in \Cref{fig:X_reflecting}.
    Without the boundary constraint, the continuous evolution would correspond to the Ornstein--Uhlenbeck process~\cite{oksendal2003stochastic}.

    First, the stochastic Frobenius--Perron operator is simulated over a time horizon of $T=2.5$ seconds, with initial condition
    \[
        v(0,x)=\mathcal N(\mu,\sigma^2),
    \]
    where $\mu=0.75$ and $\sigma=0.20$.
    As shown in \Cref{fig:v_reflecting}(a), the density drifts rightward toward $x=b$ and accumulates near the boundary due to reflection, eventually approaching a stationary distribution after the drift is balanced with the diffusion. 
    The corresponding Monte-Carlo approximation in \Cref{fig:v_reflecting}(b) is consistent with the numerical soluation of the Frobenius--Perron operator.


    Similarly, the stochastic Koopman operator is simulated from the initial condition
    \[
        u(0,x)=\mathcal N(\mu,\sigma^2),
    \]
    with $\mu=1.25$ and $\sigma=0.20$.
    The resulting evolution is shown in \Cref{fig:u_reflecting}, together with the Monte-Carlo simulation, and the two agree well.
	
    The Koopman observable converges to a constant, illustrating the forgetting behavior of ergodic systems. 
	Regardless of the initial state, the Koopman observable converges to the same value, namely the mean of $u(0,x)$ with respect to the steady-state density.

\end{example}

\begin{example}[Absorbing Boundary]\label{ex:absorbing}
    Let $\X=[a,b]$. We take
    \[
        \G=\emptyset, \qquad \Gamma=\partial\X=\{a,b\},
    \]
    and prescribe a reflecting boundary condition at $x=a$ and an absorbing boundary condition at $x=b$.
    Thus, in the notation of \Cref{assump:jump},
    \[
        \Gamma^{\rm ref}=\{a\}, \qquad \Gamma^{\rm abs}=\{b\}.
    \]
    No reset map is involved in this example.
    Trajectories are reflected at $x=a$ and removed from the system upon reaching $x=b$.
    This behavior is illustrated at \Cref{fig:X_absorbing}.

    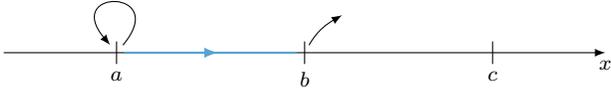
\begin{figure}
		\centerline{
            \begin{tikzpicture}
                \footnotesize
                \draw[->] (0,0) -- (8,0) node[below] {$x$};
                \node (c) at (6.5,0) {};
                \node (a) at (1.5,0) {};
                \node (b) at ($(a)!0.5!(c)$) {};
                \draw[->] (a) to[out=50, in=130, loop, looseness=15] (a) {};
                \draw[->] (b) to[out=60, in=210] ++(0.5,0.5) {};
                \draw (b |- 0, 0.15) -- ++(0,-0.3) node[below] {$b$};
                \draw (a |- 0,0.15) -- ++(0,-0.3) node[below] {$a$};
                \draw (c |- 0,0.15) -- ++(0,-0.3) node[below] {$c$};
                \draw[thick, myblue] (a) -- (b);
                \draw[->,myblue,thick] ($(a)!0.5!(b)$) -- ++(0.1,0);
            \end{tikzpicture}
		}
		\caption{\Cref{ex:absorbing}: the trajectory is reflected at $x=a$ and absorbed at $x=b$.}
        \label{fig:X_absorbing}
    \end{figure}

    By definition of an absorbing boundary, the density vanishes at $x=b$, and the Koopman observable is zero there:
    \[
        v(t,b)=0, \qquad u(t,b)=0.
    \]
    At the reflecting boundary $x=a$, the probability flux vanishes and the Koopman observable satisfies the Neumann condition:
    \begin{align}
        Y(t,a)=0, \qquad \frac{\partial u(t,a)}{\partial x}=0.
        \label{eqn:u_bc_absorbing}
    \end{align}
    Substituting these conditions into \eqref{eqn:BC_P1_ex} and \eqref{eqn:BC_P2_ex}, we verify that the weak boundary compatibility conditions are satisfied.
    In contrast to \Cref{ex:reflecting}, the total density is no longer preserved, since the outward probability flux at $x=b$ need not vanish.

    Using the same parameters and initial conditions as in \Cref{ex:reflecting}, the stochastic Frobenius--Perron operator is simulated for this mixed reflecting/absorbing boundary configuration, and the result is presented in \Cref{fig:v_absorbing}.
    Since trajectories are removed upon reaching $x=b$, the total mass of the density decreases over time, as seen from the shrinking area under the curve in \Cref{fig:v_absorbing}(a).
    The presented numerial solution is consistent with the Monte-Carlo approximation shown in \Cref{fig:v_absorbing}(b).
    The corresponding numerical simulation of the stochastic Koopman operator is shown in \Cref{fig:u_absorbing}.

    \begin{figure}
        \centering
        \subfloat[Evolution of density $v$]{\includegraphics[width=0.48\columnwidth]{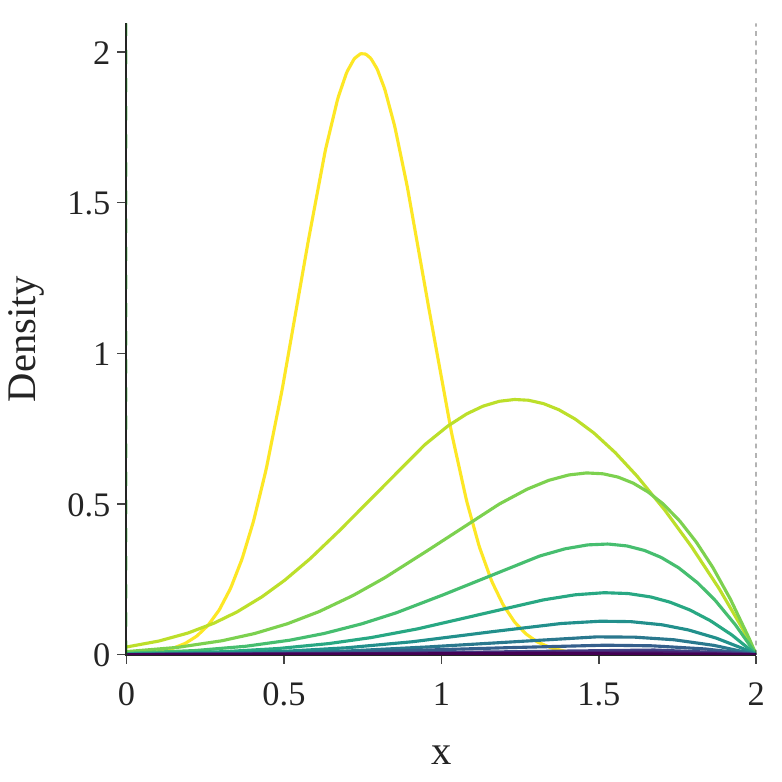}}
        \hfill
        \subfloat[Monte-Carlo simulation]{\includegraphics[width=0.48\columnwidth]{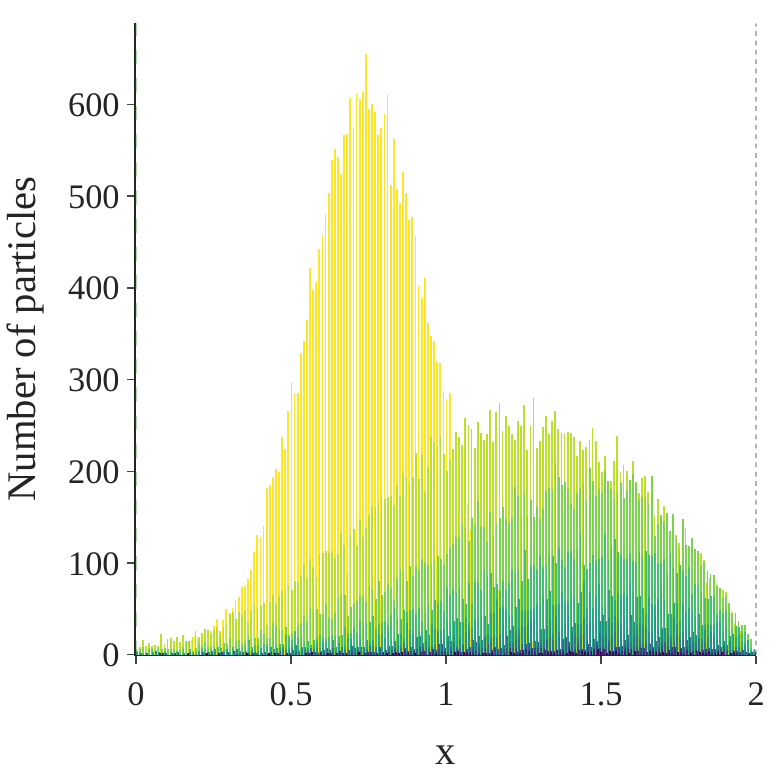}}
        \caption{\Cref{ex:absorbing}: numerical simulation of the Frobenius--Perron operator for mixed reflecting and absorbing boundaries.}
        \label{fig:v_absorbing}
    \end{figure}

    \begin{figure}
        \centering
        \subfloat[Evolution of observable $u$]{\includegraphics[width=0.48\columnwidth]{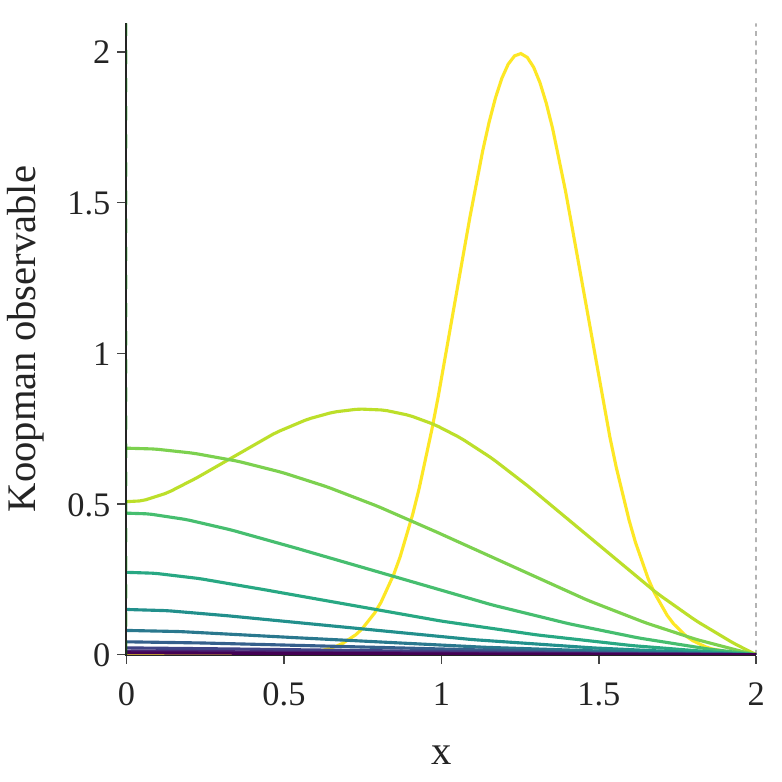}}
        \hfill
        \subfloat[Monte-Carlo simulation]{\includegraphics[width=0.48\columnwidth]{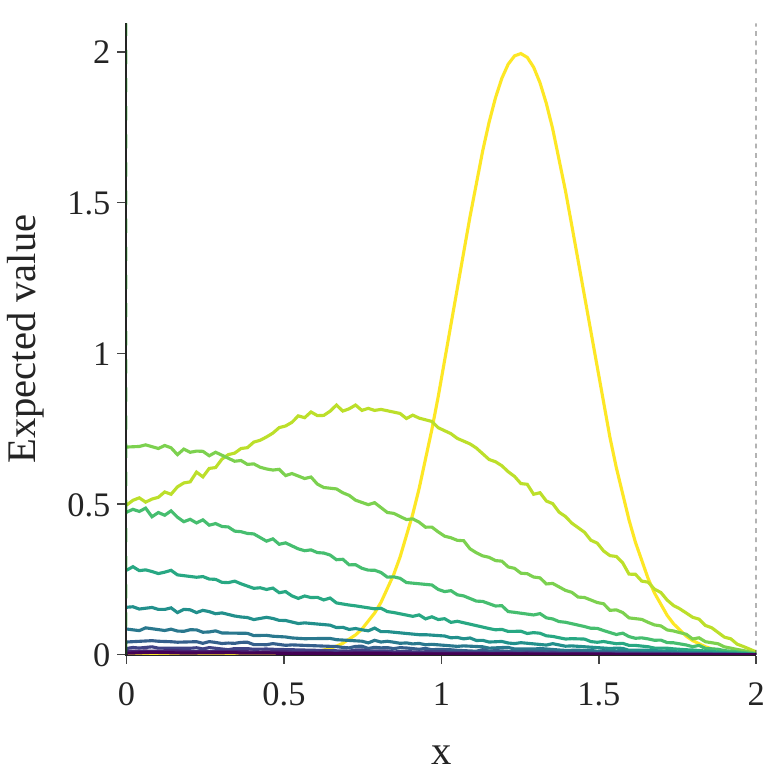}}
        \caption{\Cref{ex:absorbing}: numerical simulation of the Koopman operator for mixed reflecting and absorbing boundaries.}
        \label{fig:u_absorbing}
    \end{figure}

\end{example}

\begin{example}[Resetting Boundary]\label{ex:reset}
    Let $\X=[a,b]$. We take
    \[
        \G=\{b\}, \quad \Gamma=\{a\}, \quad \Gamma^{\rm ref}=\{a\}, \quad \Gamma^{\rm abs}=\emptyset.
    \]
    The reset map is defined by
    \[
        \Phi(b)=a,
    \]
    so that $\Phi(\G)\subset\Gamma$, consistent with \Cref{assump:jump}.
    Thus, trajectories reaching $x=b$ are instantaneously transferred to $x=a$, while $x=a$ is a reflecting boundary, as shown at \Cref{fig:X_resetting}.

    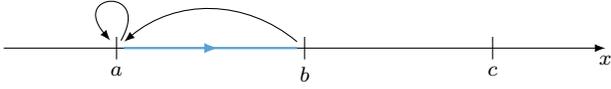
\begin{figure}
        \begin{center}
            \begin{tikzpicture}
                \footnotesize
                \draw[->] (0,0) -- (8,0) node[below] {$x$};
                \node (c) at (6.5,0) {};
                \node (a) at (1.5,0) {};
                \node (b) at ($(a)!0.5!(c)$) {};
                \draw[->] (a) to[out=60, in=130, loop, looseness=15] (a) {};
                \draw[->] (b) to[out=140, in=40] (a) {};
                \draw (b |- 0, 0.15) -- ++(0,-0.3) node[below] {$b$};
                \draw (a |- 0,0.15) -- ++(0,-0.3) node[below] {$a$};
                \draw (c |- 0,0.15) -- ++(0,-0.3) node[below] {$c$};
                \draw[thick, myblue] (a) -- (b);
                \draw[->,myblue,thick] ($(a)!0.5!(b)$) -- ++(0.1,0);
            \end{tikzpicture}
        \end{center}
		\caption{\Cref{ex:reset}: trajectories are reflected at $x=a$ and reset from $x=b$ to $x=a$.}
        \label{fig:X_resetting}
    \end{figure}

    The reset condition \eqref{eqn:BCu2} implies continuity of the Koopman observable across the jump:
    \[
        u(t,b)=u(t,a).
    \]
    At the guard $x=b$, trajectories leave the continuous domain immediately, so the interior density satisfies
    \[
        v(t,b)=0.
    \]

    At the physical boundary $x=a$, the reflecting condition imposes the Neumann boundary condition
    \begin{align}
        \frac{\partial u(t,a)}{\partial x}=0.
    \end{align}
    However, unlike a purely reflecting boundary, $x=a$ also receives probability mass from the reset $\Phi(b)=a$.
    Therefore, no separate condition $Y(t,a)=0$ is imposed.

    Substituting these conditions into \eqref{eqn:BC_P1_ex}, we obtain
    \begin{align}
        -u(t,a)Y(t,a)+u(t,b)Y(t,b)=0,
    \end{align}
    which, using $u(t,b)=u(t,a)$, reduces to
    \begin{align}
        Y(t,a)=Y(t,b).
    \end{align}
	This expresses that the outward flux through the guard at $x=b$ is exactly reinjected at its reset image $x=a$, ensuring conservation of total mass.

    Similarly, substituting into \eqref{eqn:BC_P2_ex}, we obtain
    \begin{align}
        -v(t,a)H\frac{\partial u(t,a)}{\partial x}+v(t,b)H\frac{\partial u(t,b)}{\partial x}=0,
    \end{align}
    which is satisfied since $v(t,b)=0$ and $\partial_x u(t,a)=0$.

    The numerical simulation for the Frobenius--Perron operator is shown in \Cref{fig:v_reset}.
    The density decreases as $x\to b$, reflecting the increasing likelihood of hitting the guard and being reset to $x=a$.
    Despite this redistribution, the total mass is conserved due to reinjection at $x=a$.

    The simulation results for the stochastic Koopman operator are presented in \Cref{fig:u_reset}.
    As in \Cref{ex:reflecting}, the observable converges to a constant value, and the transient response is consistent with the Monte-Carlo simulation.

    The resulting boundary conditions for three types of hybrid systems are summarized at \Cref{tab:EX1D}.

    \begin{figure}
        \centering
        \subfloat[Evolution of density $v$]{\includegraphics[width=0.48\columnwidth]{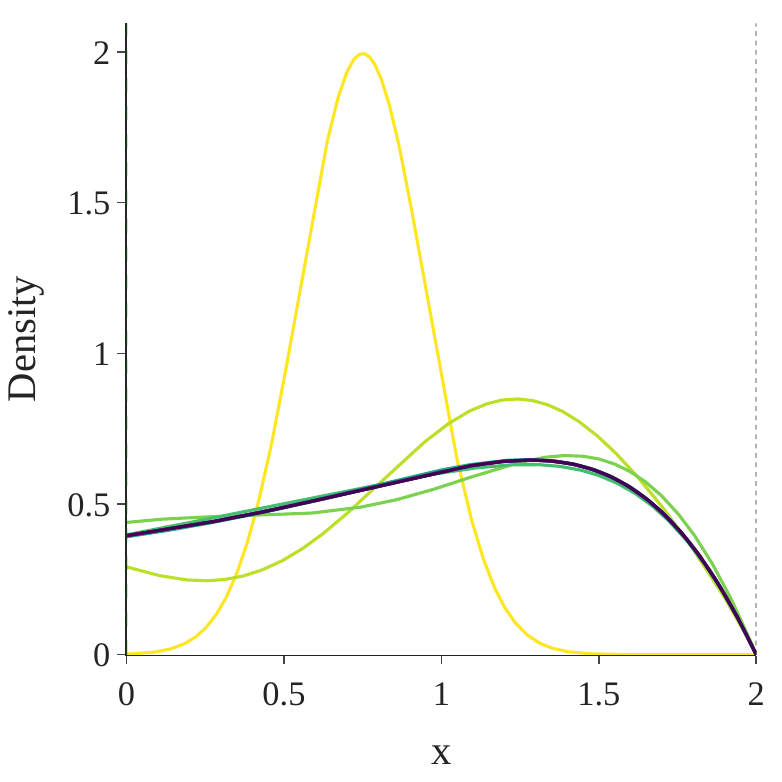}}
        \hfill
        \subfloat[Histogram of Monte-Carlo simulation]{\includegraphics[width=0.48\columnwidth]{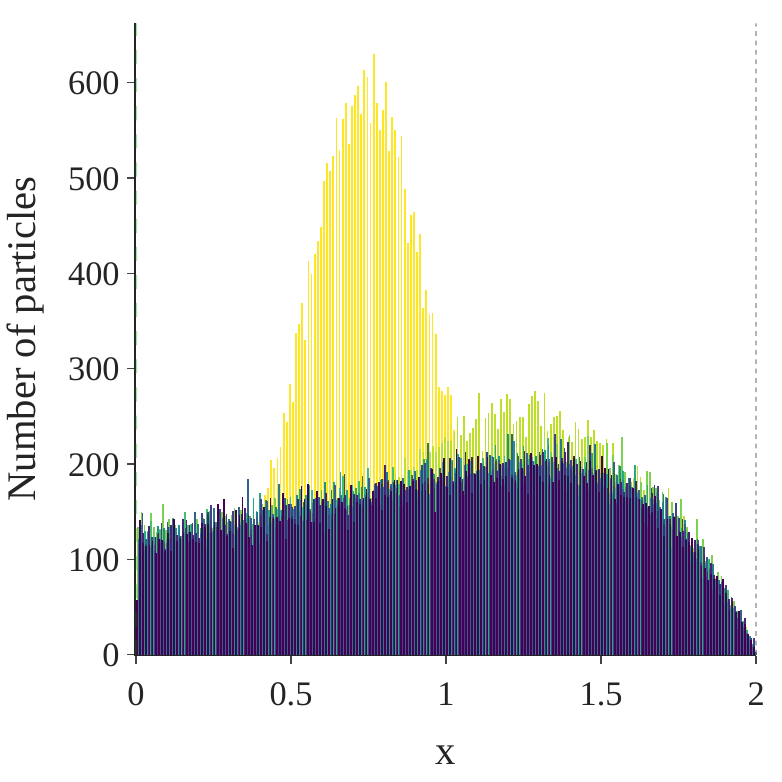}}
        \caption{\Cref{ex:reset}: numerical simulation of the Frobenius--Perron operator for resetting boundary.}
        \label{fig:v_reset}
    \end{figure}

    \begin{figure}
        \centering
        \subfloat[Evolution of observable $u$]{\includegraphics[width=0.48\columnwidth]{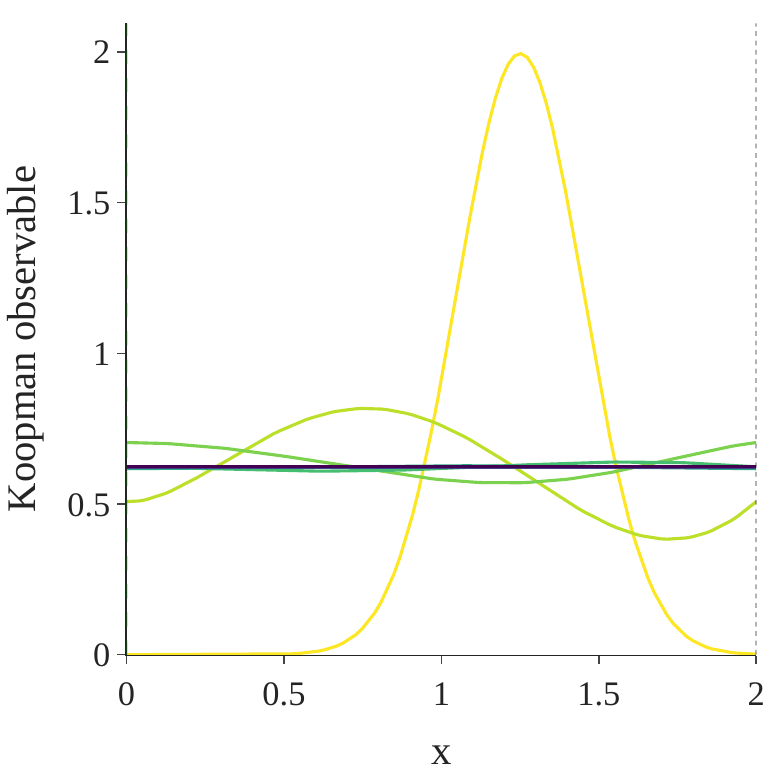}}
        \hfill
        \subfloat[Monte-Carlo simulation]{\includegraphics[width=0.48\columnwidth]{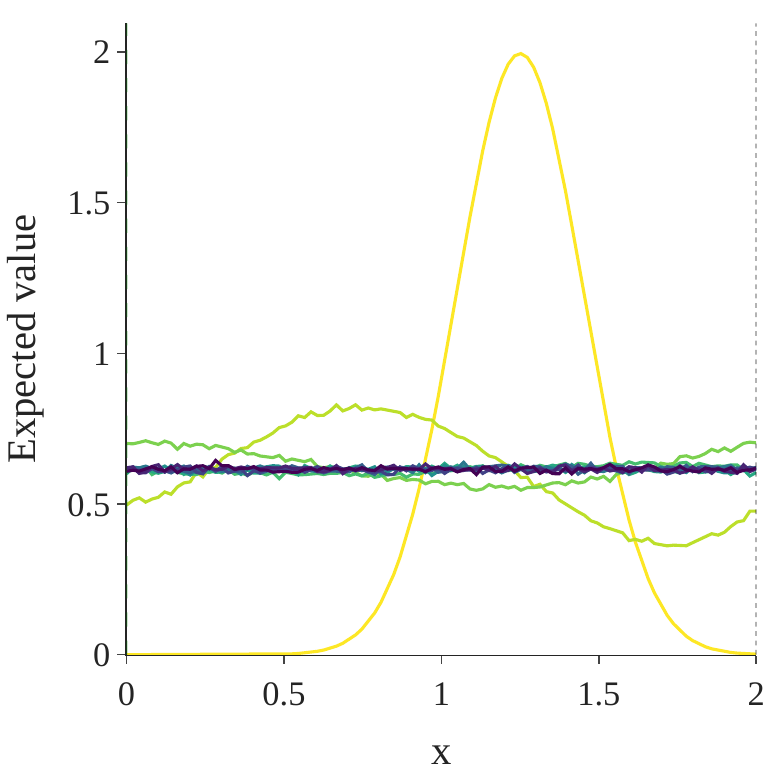}}
        \caption{\Cref{ex:reset}: numerical simulation of the Koopman operator for resetting boundary.}
        \label{fig:u_reset}
    \end{figure}

\end{example}

\begin{table}
    \caption{Summary of boundary conditions}
    \label{tab:EX1D}
    \begin{center}
        \begin{tabular}{ccc}
            \toprule
            & $x=a$ & $x=b$ \\ \midrule
            Example 1 
            & $Y(a)=0$, $\partial_x u(a)=0$ 
            & $Y(b)=0$, $\partial_x u(b)=0$ \\
            & (reflecting) & (reflecting) \\ \midrule
            Example 2 
            & $Y(a)=0$, $\partial_x u(a)=0$ 
            & $u(b)=0$, $v(b)=0$ \\
            & (reflecting) & (absorbing) \\ \midrule
            Example 3 
            & $Y(a)=Y(b)$, $\partial_x u(a)=0$ 
            & $u(b)=u(a)$, $v(b)=0$ \\
            & (reset image with reflection) & (guard) \\ \bottomrule
        \end{tabular}
    \end{center}
\end{table}

\section{Stochastic Hybrid Systems with Multiple Modes}\label{sec:multi_mode}

In this section, we extend the results of \Cref{sec:SHS1} to the general case of stochastic hybrid systems with multiple discrete modes. Each mode is associated with a distinct continuous state space, stochastic flow, and boundary behavior, while discrete transitions couple the modes through boundary-induced resets.

\subsection{Problem Formulation}

As introduced in \Cref{sec:SHS}, let $\Q=\{1,2,\ldots,N_Q\}$ be a finite set of discrete modes.
For each $q\in\Q$, the continuous state evolves on an $n_q$-dimensional Riemannian manifold $\X_q$ with boundary $\partial\X_q$ and metric $g_q$.
The multi-mode system can be viewed as a collection of single-mode systems coupled through reset maps across modes.

\begin{assumption}[Continuous Flow]\label{assump:flow_q}
For each $q\in\Q$, the continuous dynamics are given by
\begin{align}
    dx = X_q(x)\,dt + \sum_{i=1}^{n_q} \sigma_q E_{q,i}(x) dW_i,
    \label{eqn:SDEq}
\end{align}
where $X_q\in\mathfrak{X}(\X_q)$ is the drift vector field, $\sigma_q>0$, and $\{E_{q,1}(x),\ldots,E_{q,n_q}(x)\}$ is an orthonormal frame of $\X_q$.
\end{assumption}

As in the single-mode case, the diffusion is isotropic with respect to the Riemannian metric.
The boundary behavior at each mode follows the same local definitions of reflecting and absorbing boundaries given in \Cref{def:boundary_single}, for example with $Y$ replaced by $Y_q$.

\begin{assumption}[Boundary Decomposition and Reset Structure]\label{assump:jump_q}
For each $q\in\Q$, the boundary is decomposed as
\[
    \partial \X_q = \G_q \cup \Gamma_q, 
    \qquad 
    \G_q\cap\Gamma_q=\emptyset.
\]
Furthermore:

\begin{itemize}

    \item \textbf{Reset on the guard.}
    Upon intersecting with $\G_q$, the hybrid state resets according to
    \begin{align}
        (x^+,q^+) = \Phi(x,q), \qquad x\in\G_q,
        \label{eqn:reset_q}
    \end{align}
    where $\Phi(\G_q)\subset \Gamma_{q^+}\times\{q^+\}$.

    \item \textbf{Guard trace condition.}
    As in \Cref{assump:jump}, trajectories do not remain on $\G_q$, and the density satisfies
    \begin{align}
        v(t,x,q)=0, \qquad x\in\G_q.
        \label{eqn:BCv_guard}
    \end{align}

    \item \textbf{Guard--flux compatibility.}
    Probability mass leaving mode $q$ through $\G_q$ is transferred by the reset map to $\Phi(\G_q)\subset \Gamma_{q^+}$ without accumulation or loss.

    \item \textbf{Physical boundary on $\Gamma_q$.}
    The set $\Gamma_q$ is partitioned as
    \[
        \Gamma_q = \Gamma_q^{\rm ref}\cup \Gamma_q^{\rm abs}, 
        \qquad
        \Gamma_q^{\rm ref}\cap\Gamma_q^{\rm abs}=\emptyset,
    \]
    where the reflecting and absorbing boundary conditions are imposed as in \Cref{def:boundary_single}.

\end{itemize}
\end{assumption}
Each mode inherits the local boundary behavior defined in \Cref{sec:SHS1}, 
while the reset map couples the boundary fluxes across different modes.

\subsection{Stochastic Koopman/Frobenius--Perron Operator}

Next, we develop the stochastic Koopman operator and the Frobenius--Perron operator.

\begin{theorem}\label{thm:Kq}
    Consider the stochastic hybrid system satisfying \Cref{assump:flow_q,assump:jump_q}. 
    Let
    \[
        u(t,x,q)=\mathcal K_t f(x,q),
    \]
    where $f:\H\to\Re$ belongs to the domain of the generator. Then the stochastic Koopman observable satisfies
    \begin{align}
        \frac{\partial u(t,x,q)}{\partial t} &=\mathcal A_q u(t,x,q), \qquad x\in\X_q\setminus \G_q,\label{eqn:dotuq}\\
        u(t,x,q) &=u(t,\Phi(x,q)), \qquad x\in\G_q,\label{eqn:BCuq}
    \end{align}
    with the initial condition
    \[
        u(0,x,q)=f(x,q).
    \]
    The infinitesimal generator $\mathcal A_q$ is given on its domain by
    \begin{align}
        \mathcal A_q u(t,x,q) = \mathcal L_{X_q}u(t,x,q)+H_q\Delta_q u(t,x,q), \label{eqn:Aq}
    \end{align}
    where $\Delta_q$ is the Laplace--Beltrami operator on $\X_q$ with metric $g_q$, and $H_q=\frac12\sigma_q^2>0$.
\end{theorem}

\begin{proof}
    In the interior of $\X_q$, the result follows directly from \Cref{thm:K2}, which yields \eqref{eqn:dotuq} and \eqref{eqn:Aq}. 
    Next, \eqref{eqn:BCuq} follows from the reset rule \eqref{eqn:reset_q}: if a trajectory reaches a point $(x,q)\in\G_q$, then it is instantaneously mapped to $\Phi(x,q)$, and hence the subsequent evolution coincides with that initiated from $\Phi(x,q)$. Therefore the corresponding Koopman observables agree, which gives \eqref{eqn:BCuq}.
\end{proof}

Let $v(t,x,q)\in L^1(\X_q)$ denote the mode-conditioned density, satisfying
\begin{equation}    
    \sum_{q\in\Q}\int_{\X_q} v(t,x,q)\,d\mu_q(x)=1. \label{eqn:vq_normalization}
\end{equation}
The pairing, or equivalently the duality, between observables and densities is defined by \eqref{eqn:pair}. 
Similar to \Cref{thm:P2}, the generator of the Frobenius--Perron operator for multiple modes is obtained as follows.

\begin{theorem}\label{thm:Pq}
    Consider the stochastic hybrid system satisfying \Cref{assump:flow_q,assump:jump_q}. 
    Let $g:\H\to\Re$ be an initial density and define
    \[
        v(t,x,q)=\mathcal P_t g(x,q),
    \]
    where $v(t,x,q)\in L^1(\X_q)$ satisfies \eqref{eqn:vq_normalization}.
    Then the stochastic Frobenius--Perron evolution satisfies
    \begin{align}
        \frac{\partial v(t,x,q)}{\partial t}
        =
        \mathcal A_q^* v(t,x,q),
        \qquad x\in \X_q\setminus \G_q,
        \label{eqn:dotvq}
    \end{align}
    with initial condition
    \[
        v(0,x,q)=g(x,q).
    \]

    The adjoint generator $\mathcal A_q^*$ is given on its domain by
    \begin{align}
        \mathcal A_q^* v(t,x,q) &= -\mathcal L_{X_q}v(t,x,q) - v(t,x,q)\,\mathrm{div}_{\mu_q}(X_q)\nonumber\\
                                & \quad +H_q\Delta_q v(t,x,q), \label{eqn:A*q}
    \end{align}
    where $H_q=\frac12\sigma_q^2$. 
    Equivalently,
    \begin{align}
        \mathcal A_q^* v(t,x,q)
        =
        -\div_{\mu_q}(Y_q(t,x)),
        \label{eqn:A*q_div}
    \end{align}
    where
    \begin{align}
        Y_q(t,x)
        =
        v(t,x,q)X_q(x)-H_q\nabla v(t,x,q).
        \label{eqn:Yq}
    \end{align}

    Moreover, adjointness of the Koopman and Frobenius--Perron generators requires the boundary compatibility conditions
    \begin{gather}
        \sum_{q\in\Q}\int_{\partial \X_q}
        u(t,x,q)\,(Y_q(t,x)\cdot N_q(x))\,\nu_q
        =0,
        \label{eqn:BC_P1q}\\
        \sum_{q\in\Q}\int_{\partial \X_q}
        v(t,x,q)H_q(\nabla u(t,x,q)\cdot N_q(x))\,\nu_q
        =0,
        \label{eqn:BC_P2q}
    \end{gather}
    for every $u(t,x,q)$ in the domain of the Koopman generator, namely smooth in the interior and satisfying the admissibility conditions induced by \Cref{assump:jump_q} and \eqref{eqn:BCuq}.

    Here $N_q$ denotes the outward unit normal vector along $\partial\X_q$, and
    \[
        \nu_q=\iota_q^*(i_{N_q}\mu_q)
    \]
    is the induced $(n_q-1)$-form on $\partial\X_q$, where $\iota_q:\partial\X_q\hookrightarrow\X_q$ is the inclusion.
\end{theorem}

\begin{proof}
    For each $q\in\Q$, let
    \[
        u_q(t,x)=u(t,x,q), \qquad v_q(t,x)=v(t,x,q).
    \]
    Applying the results of \Cref{thm:P2} to each mode and summing over $q\in\Q$ yields
    \begin{align}
        \sum_{q\in\Q}\pair{v_q,\mathcal A_q u_q} &= \sum_{q\in\Q}\pair{\mathcal A_q^*v_q,u_q} 
        +\sum_{q\in\Q}\int_{\partial\X_q} u_q(Y_q\cdot N_q)\,\nu_q\nonumber\\
                                                 & \quad +\sum_{q\in\Q}\int_{\partial\X_q} v_qH_q(\nabla u_q\cdot N_q)\,\nu_q, \label{eqn:thm_tmp1}
    \end{align}
    with \eqref{eqn:A*q}.

    Now, we focus on the last two integrals on the boundary $\partial \X_q$. 
    First, consider the last integral $\int_{\partial\X_q} v_qH_q(\nabla u_q\cdot N_q)\,\nu_q$.
    By \Cref{assump:jump_q}, $\partial\X_q=\G_q\cup\Gamma_q^{\rm ref}\cup\Gamma_q^{\rm abs}$.
    On $\G_q$, \eqref{eqn:BCv_guard} gives $v_q=0$.
    On $\Gamma_q^{\rm ref}$, \eqref{eqn:BC_uv_reflecting_single} gives $\nabla u_q\cdot N_q=0$.
    On $\Gamma_q^{\rm abs}$, \eqref{eqn:BC_uv_absorbing_single} gives $v_q=0$.
    Hence the integrand vanishes on each part of $\partial\X_q$, showing \eqref{eqn:BC_P2q}.

    Next, consider $\int_{\partial\X_q} u_q(Y_q\cdot N_q)\,\nu_q$.
    On $\Gamma_q^{\rm ref}$, \eqref{eqn:BC_uv_reflecting_single} gives $Y_q\cdot N_q=0$.
    On $\Gamma_q^{\rm abs}$, \eqref{eqn:BC_uv_absorbing_single} gives $u_q=0$.
    Thus the contribution from the physical boundary vanishes.
    The remaining guard contribution need not vanish pointwise; however, by the guard--flux compatibility in \Cref{assump:jump_q}, the outgoing flux through each guard is balanced by the corresponding incoming flux on the reset image, so the net contribution of all guard/reset-image pairs is zero. Therefore
    \[
        \sum_{q\in\Q}\int_{\partial\X_q} u_q(Y_q\cdot N_q)\,\nu_q=0,
    \]
    which yields \eqref{eqn:BC_P1q}.

    Substituting these two identities into \eqref{eqn:thm_tmp1} gives
    \[
        \sum_{q\in\Q}\pair{v_q,\mathcal A_q u_q}
        =
        \sum_{q\in\Q}\pair{\mathcal A_q^*v_q,u_q}.
    \]
    Since this holds for every admissible Koopman observable $u$, the density satisfies
    \[
        \frac{\partial v(t,x,q)}{\partial t}=\mathcal A_q^*v(t,x,q),
        \qquad x\in\X_q\setminus\G_q,
    \]
    with initial condition $v(0,x,q)=g(x,q)$. This yields \eqref{eqn:dotvq}.
\end{proof}

Next, we present a stochastic hybrid system with multiple modes, namely a hybrid system on a torus in \Cref{sec:ex_Torus}.

\subsection{Example: Torus}\label{sec:ex_Torus}

Consider a stochastic process on the torus $\mathbb{T}=\Sph^1\times\Sph^1$, where $\Sph^1=\{q\in\Re^2\mid \|q\|=1\}$ denotes the unit circle.
We regard $\mathbb{T}$ as a two-dimensional manifold embedded in $\Re^3$, and parameterize a point $x\in\mathbb{T}$ by
\begin{align}
    x(\theta, \psi)
    = \begin{bmatrix}
        (R + r\cos\theta)\cos\psi\\
        (R + r\cos\theta)\sin\psi\\
        r\sin\theta
    \end{bmatrix},
\end{align}
where $\theta,\psi\in[0,2\pi)$ and $R, r>0$ are the major radius and the minor radius.
Here, $\psi$ corresponds to the rotation angle about the vertical axis at the center of the torus, and $\theta$ is the rotation angle within the smaller circle obtained by fixing $\psi$. 
Throughout this section, it is considered that any function defined on $\mathbb{T}$ is periodic when represented using the coordinates $\theta,\psi$.
For example, given $f:\mathbb{T}\rightarrow\Re$, we have $f(\theta,\psi)=f(\theta + 2\pi m, \psi + 2\pi n )$ for any integer $m,n$.

\paragraph{Geometry of Torus}
The metric induced from $\Re^3$ is
\[
    g_{\theta\theta}=r^2,\quad
    g_{\theta\psi}=g_{\psi\theta}=0,\quad
    g_{\psi\psi}=(R+r\cos\theta)^2,
\]
and the volume form is $\mu = r(R+r \cos\theta) d\theta\wedge d\psi$.
The orthonormal frame is
\begin{align*}
    E_\theta = \frac{1}{r}\partial_\theta, \quad
    E_\psi = \frac{1}{R+r\cos\theta}\partial_\psi.
\end{align*}
For $f\in C^1(\mathbb{T})$ and $X = X_\theta \partial_\theta + X_\psi\partial_\psi \in \mathfrak{X}(\mathbb{T})$, the gradient, divergence, and Laplace-Beltrami operator on $\mathbb{T}$ are given by
\begin{align}
    \nabla f & =\frac{1}{r^2} \frac{\partial f }{\partial\theta} \partial_\theta + \frac{1}{(R+r\cos\theta)^2} \deriv{f}{\psi}\partial_\psi,\label{eqn:grad_torus}\\
    \mathrm{div} X & = \deriv{X_\theta}{\theta} + \deriv{X_\psi}{\psi} - \frac{r\sin\theta}{R + r\cos\theta} X_\theta,\label{eqn:div_torus}\\
   \Delta f &= 
     \frac{1}{r^{2}}\,\frac{\partial^{2} f}{\partial\theta^{2}}
    + \frac{1}{(R+r\cos\theta)^{2}}\,\frac{\partial^{2} f}{\partial\psi^{2}}
    - \frac{\sin\theta}{r(R+r\cos\theta)}\frac{\partial f}{\partial\theta}. \label{eqn:laplace_torus}
\end{align}

\paragraph{Stochastic Hybrid Systems}

Here we formulate a stochastic hybrid system on the torus where the jump occurs at the one-dimensional surface defined by $\psi=\pi$. 
Since the jump is triggered on the interior hypersurface $\psi=\pi$, this example does not directly satisfy the standing assumption that the guard lies on the boundary of the continuous state space.
To address this, we decompose $\mathbb{T}$ into two parts:
\[
    \X_1=\{(\theta,\psi)\mid \psi\in[0,\pi)\},\quad
    \X_2=\{(\theta,\psi)\mid \psi\in[\pi,2\pi)\}.
\]
with artificial boundary components, as shown at \Cref{fig:torus_decomp}. 
As such there are two discrete modes $\Q=\{1,2\}$.
Throughout this section, the subscript 1, 2, or $q$ of a variable indicates that the variable is associated with the specified mode. 

For both modes $q=1,2$, let the drift vector field be $X(\theta, \psi, q) = X_\theta(\theta, \psi) \partial_\theta + X_\psi(\theta, \psi) \partial_\psi$ for $X_\theta,X_\psi\in C^\infty(\mathbb{T})$. 
According to \Cref{assump:flow_q}, the diffusion vector fields are chosen as the orthonormal frame so that the drift-free part becomes Brownian motion on the torus.
The resulting continuous flow for $q\in\{1,2\}$ is described by
\begin{align}
    d\theta & = X_\theta(\theta,\psi) dt  + \sigma\frac{1}{r} dW_\theta,\\
    d\psi & = X_\psi(\theta,\psi) dt + \sigma\frac{1}{R+ r\cos\theta} dW_\psi,
\end{align}
for $\sigma > 0$.

\begin{figure}
    \footnotesize\selectfont
    \begin{center}
        \begin{tikzpicture}

            \def\r{1.7}
            \def\shift{1.4} 

            \coordinate (leftT)  at (-\r,0);
            \coordinate (rightT) at (\r,0);

            \draw (leftT)
                arc[start angle=180, end angle=0, radius=\r]
                node[pos=0.5, below=3pt] {$\X_1$};

            \fill (rightT) circle (2pt);
            \draw[fill=white] (leftT) circle (2pt);

            \node[right=4pt] at (rightT) {$\Sigma_1^+$};
            \node[left=4pt] at (rightT) {$\psi=0^+$};
            \node[left=4pt] at (leftT)  {$\Sigma_1^-$};
            \node[right=4pt] at (leftT)  {$\psi=\pi^-$};

            \coordinate (leftB)  at (-\r,-\shift);
            \coordinate (rightB) at (\r,-\shift);

            \draw (rightB)
                arc[start angle=0, end angle=-180, radius=\r]
                node[pos=0.5, above=3pt] {$\X_2$};

            \fill (leftB) circle (2pt);
            \draw[fill=white] (rightB) circle (2pt);

            \node[right=4pt] at (rightB) {$\Sigma_2^-$};
            \node[left=4pt] at (rightB) {$\psi=2\pi^-$};
            \node[left=4pt] at (leftB)  {$\Sigma_2^+$};
            \node[right=4pt] at (leftB)  {$\psi=\pi^+$};

            \draw[->, dashed, shorten <=3pt, shorten >=3pt]
                (leftT) to[out=-30, in=-150]
                node[midway, above=1pt] {$\Phi_1^-$}
                (rightT);
            \draw[->, dashed, shorten <=3pt, shorten >=3pt]
                (leftB) to[out=30, in=-150]
                node[midway, below=1pt] {$\Phi_2^+$}
                (rightT);
            \draw[<->, double, shorten <=2pt, shorten >=2pt]
                (rightT)  --
                node[midway, right=0pt] {$\Sigma_1^+\sim\Sigma_2^-$}
                (rightB);
        \end{tikzpicture}
    \end{center}
	\caption{Decomposition of the torus (solid), with periodic identification between $\Sigma_1^+$ and $\Sigma_2^-$ (double) and reset maps at $\Sigma_1^-$ and $\Sigma_2^+$ (dashed); only the $\psi$-component is shown.}\label{fig:torus_decomp}

\end{figure}
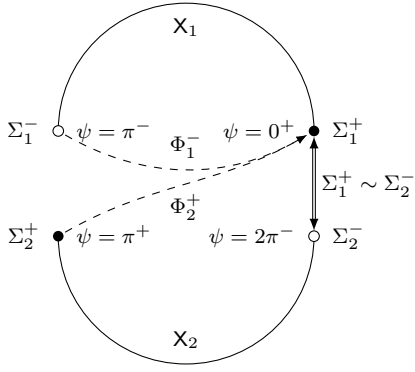
To express the interface conditions, we use one-sided traces:
\begin{align}
    \Sigma_1^+ &= \{(\theta,\psi)\in\partial\X_1 \mid \psi=0^{+}\},\\
    \Sigma_1^- &= \{(\theta,\psi)\in\partial\X_1 \mid \psi=\pi^{-}\},\\
    \Sigma_2^+ &= \{(\theta,\psi)\in\partial\X_2 \mid \psi=\pi^{+}\},\\
    \Sigma_2^- &= \{(\theta,\psi)\in\partial\X_2 \mid \psi=2\pi^{-}\}.
\end{align}
The boundary components are classified as:
\begin{itemize}
    \item Identification interface: $\Sigma_1^+ \sim \Sigma_2^-$ (periodicity)
    \item Guards: $\G_1=\Sigma_1^-$,\quad $\G_2=\Sigma_2^+$
\end{itemize}
For the two-sided jump at the guard $\psi=\pi$, define the reset map as
$\Phi_1^-:\Sigma_1^-\to\Sigma_1^+$ and $\Phi_2^+:\Sigma_2^+\to\Sigma_1^+$ by
\begin{align}
    \Phi_1^-(\theta,\pi^{-},1) &= (\theta+\pi,0^{+},1),\\
    \Phi_2^+(\theta,\pi^{+},2) &= (\theta+\pi,0^{+},1), \label{eq:phi_pi_twosided}
\end{align}
i.e., the hybrid state jumps from $\psi=\pi$ to $\psi=0$.
But, across the jump, the angle $\theta$ is shifted by $\pi$. 
This example slightly extends \Cref{assump:jump_q} by allowing the reset image to lie on an identification interface, rather than on a physical boundary.

\paragraph{Stochastic Koopman Operator}

Let $u_q(t,\theta,\psi)$ denote the Koopman observable in mode $q$.
From \eqref{eqn:Aq}, the generator for the Koopman operator is
\begin{align}
    \mathcal{A}_q u_q & = X_\theta \deriv{u_q}{\theta} + X_\psi \deriv{u_q}{\psi} + H \Delta u_q,
\end{align}
for any $q\in\Q$, 
where the expression for the Laplace-Beltrami operator is presented at \eqref{eqn:laplace_torus}.
Then, in the interior $\X_q \setminus \G_q$, the observable evolves according to \eqref{eqn:dotuq}. 

The boundary conditions follow directly from identification and reset.
First, the periodic interface provides
\begin{align}
    u_1(t,\theta,0^{+}) &= u_2(t,\theta,2\pi^{-}) \label{eqn:K_BC_periodic}.
\end{align}
On the guards, the boundary condition \eqref{eqn:BCuq} yields
\begin{align}
    u_1(t,\theta,\pi^{-}) &= u_1(t,\theta+\pi,0^{+}), \label{eqn:K_BC_pi_minus}\\
    u_2(t,\theta,\pi^{+}) &= u_1(t,\theta+\pi,0^{+}). \label{eqn:K_BC_pi_plus}
\end{align}

\paragraph{Stochastic Frobenius--Perron Operator}

Let $v_q(t,\theta,\psi)$ denote the probability density for the mode $q$. 
According to \eqref{eqn:Yq}, the vector field $Y_q\in\mathfrak{X}(\X_q)$ is given by
\begin{align*}
    Y_q & = \left(v_q X_\theta - \frac{H}{r^2} \deriv{v_q}{\theta}\right) \partial_\theta \\
        & \quad + \left(v_q X_\psi - \frac{H}{(R+r\cos\theta)^2} \deriv{v_q}{\psi}\right) \partial_\psi.
\end{align*}
Using the expression of the divergence on the torus given by \eqref{eqn:div_torus}, the generator of the Frobenius--Perron operator is
\begin{align}
    \mathcal{A}_q^* v_q & = -\deriv{(v_qX_\theta)}{\theta} - \deriv{(v_qX_\psi)}{\psi} +\frac{r\sin\theta}{R+r\cos\theta} v_qX_\theta \nonumber\\
                        &  +\frac{H}{r^2} \frac{\partial^2 v_q}{\partial \theta^2} + \frac{H}{(R+r\cos\theta)^2} \frac{\partial^2 v_q}{\partial \psi^2} -\frac{H\sin\theta}{r(R+r\cos\theta)}\deriv{v_q}{\theta}. \label{eqn:A*_torus}
\end{align}

The corresponding expression of the generator in the conservative form is
\begin{align}
	\mathcal{A}_q^* v_q & = -\frac{1}{J} \left\{ \frac{\partial (J v_q X_\theta)} {\partial\theta} + \frac{\partial (J v_q X_\psi)}{\partial\psi}  \right\} 
	\nonumber\\
	& + H \frac{1}{J} \left\{ \frac{\partial}{\partial\theta} \left(\frac{J}{g_{\theta\theta}} \frac{\partial v_q}{\partial \theta}\right)  + \frac{\partial}{\partial\psi} \left(\frac{J}{g_{\psi\psi}} \frac{\partial v_q}{\partial \psi}\right) \right\} \label{eqn:A*_torus_cons}
\end{align}
where $ J = \sqrt{\mathrm{det}g}  = r(R+r\cos\theta)$.

Next, we identify the boundary conditions from \eqref{eqn:BC_P1q} and \eqref{eqn:BC_P2q}.
The outward unit normal vector at the boundary is $N_q = \mp \frac{1}{\sqrt{g_{\psi\psi}}}\partial_\psi = \mp \frac{1}{R+r\cos\theta}\partial_\psi$, which is negative at $\Sigma_1^+$ and $\Sigma_2^+$, and positive at $\Sigma_1^-$ and $\Sigma_2^-$. 
The induced differential form at the boundary is $\nu_q = \iota^* i_{N_q} \mu_q = \mp rd\theta$, which has the same sign as $N_q$. 
Therefore, 
\begin{align}
    (Y_q\cdot N_q)\nu_q & = \left\{ r(R+r\cos\theta) v_q X_\psi -\frac{Hr}{R+r\cos\theta} \deriv{v_q}{\psi} \right\} d\theta\nonumber\\
                        & \triangleq \mathbf{Y}_q(t,\theta,\psi) d\theta. \label{eqn:YNnu}
\end{align}
for $\mathbf{Y}_q:\Re\times\mathbb{T}\rightarrow\Re$.

Since the boundary at $\psi = 0^+, \pi^+$ is negatively oriented, and the boundary at $\psi=\pi^-, 2\pi^-$ is positively oriented, 
\eqref{eqn:BC_P1q} yields
\begin{gather*}
    \int_0^{2\pi} -u_1(\theta, 0^+) \mathbf{Y}_1(\theta, 0^+) 
    + u_1(\theta, \pi^-) \mathbf{Y}_1(\theta, \pi^-)\, d\theta\\
    +\int_0^{2\pi} -u_2(\theta, \pi^+) \mathbf{Y}_2(\theta, \pi^+) 
    + u_2(\theta, 2\pi^-) \mathbf{Y}_2(\theta, 2\pi^-)\, d\theta = 0,
\end{gather*}
where we suppressed the dependency on time for brevity. 
Substituting the boundary conditions given by \eqref{eqn:K_BC_periodic}--\eqref{eqn:K_BC_pi_plus}, and changing the integration variable for the second and the third term with the periodicity of $u_1$ and $u_2$, this is rearranged into
\begin{gather*}
    \int_0^{2\pi} u_1(\theta, 0^+) \big\{ -\mathbf{Y}_1(\theta, 0^+) 
        +\mathbf{Y}_1(\theta-\pi, \pi^-)\\
        - \mathbf{Y}_2(\theta-\pi, \pi^+)
    + \mathbf{Y}_2(\theta, 2\pi^-) \} d\theta = 0.
\end{gather*}
Since this should be satisfied for an arbitrary $u_1(\theta, 0^+)$, we obtain
\begin{gather}
    \mathbf{Y}_1(\theta, 0^+)  - \mathbf{Y}_1(\theta-\pi, \pi^-) = \mathbf{Y}_2(\theta, 2\pi^-) - \mathbf{Y}_2(\theta-\pi, \pi^+), \label{eqn:torus_BC_Y}
\end{gather}
which expresses conservation of probability under the combined effect of periodic identification and reset.
In particular, the total probability current leaving each guard is transported by the reset map and reappears as incoming flux at its image, so the total probability mass on the torus is preserved.
Equivalently, the probability current is continuous under the pushforward induced by the reset map.

Next, as discussed in \Cref{ex:reset}, since the state immediately jumps at $\psi = \pi^+,\pi^-$, we have
\begin{align}
    v_1(\theta, \pi^-) = v_2(\theta, \pi^+) = 0.\label{eqn:torus_BC_v_pi}
\end{align}
Similar to \eqref{eqn:YNnu}, we can show
\begin{align}
(\nabla u_q \cdot N_q)\nu_q & = \frac{r}{R+r\cos\theta} \deriv{u_q}{\psi} d\theta.
\end{align}
Substituting this into \eqref{eqn:BC_P2q} together with \eqref{eqn:torus_BC_v_pi},  we obtain
\begin{gather}
    v_1(\theta, 0^+) = v_2(\theta, 2\pi^-), \label{eqn:torus_BC_v_0}
\end{gather}
which is interpreted as the continuity of the density over $\psi=0^+$ and $\psi=2\pi^-$. 

In summary, the density evolves according to \eqref{eqn:A*_torus} in the interior, together with the boundary conditions \eqref{eqn:torus_BC_Y}, \eqref{eqn:torus_BC_v_pi}, and \eqref{eqn:torus_BC_v_0} induced by periodic identification and reset.

\subsection{Numerical Simulation}

We discretize the governing equations using a finite-volume method on a uniform $(\psi,\theta)$ grid.
This approach preserves total mass, accommodates variable coefficients through face-centered fluxes, and naturally incorporates the boundary conditions.
Specifically, to approximate the continuous stochastic flow, we employ the conservative form of the generator \eqref{eqn:A*_torus_cons}, which consists of both advective and diffusive contributions in each coordinate direction.

Let $(i,j)$ denote the grid indices corresponding to $(\psi,\theta)$.
Along the $\psi$-direction, the advective and diffusive fluxes at the interface $i+\tfrac12$ are approximated by
\begin{align*}
	F^\text{adv}_{i+\tfrac12,j} &\approx (Jv_qX_\psi)_{i+\tfrac12,j},\\
	F^\text{diff}_{i+\tfrac12,j} &\approx -H\,J_{i+\tfrac12,j}\,g^{\psi\psi}_{i+\tfrac12,j} \frac{v_{i+1,j}-v_{i,j}}{\Delta\psi},
\end{align*}
where the advective flux is evaluated using an upwind reconstruction based on the sign of the normal velocity $X_{\psi,i+\tfrac12,j}$:
\begin{equation*}
F^\text{adv}_{i+\tfrac12,j} =
\begin{cases}
    J_{i+\tfrac12,j} X_{\psi,i+\tfrac12,j}\, v^{\text{L}}_{i+\tfrac12,j}, & X_{\psi,i+\tfrac12,j}>0,\\[4pt]
    J_{i+\tfrac12,j} X_{\psi,i+\tfrac12,j}\, v^{\text{R}}_{i+\tfrac12,j}, & X_{\psi,i+\tfrac12,j}<0,
\end{cases}
\end{equation*}
where $v^{\text{L}}_{i+\tfrac12,j}$ and $v^{\text{R}}_{i+\tfrac12,j}$ denote the reconstructed left and right interface states.
The resulting total flux is 
\[
	F_{i+\tfrac12,j} = F^\text{adv}_{i+\tfrac12,j} + F^\text{diff}_{i+\tfrac12,j}.
\]

An analogous construction yields the flux $G_{i,j+\tfrac12}$ in the $\theta$-direction using $X_\theta$ and $g^{\theta\theta}$.
All coefficients, such as $J$, $g^{ii}$, and $X_i$, are evaluated at face centers via averaging.

For the advective fluxes in both coordinate directions, we employ a fifth-order weighted essentially non-oscillatory (WENO5) reconstruction to compute the interface states~\cite{jiang1996efficient,liu1994weighted}.
This provides high-order accuracy in smooth regions while avoiding spurious oscillations near steep gradients.

Using these fluxes, the implicit time-stepping scheme is given by
\begin{align*}
	v_{i,j}^{n+1} &= v_{i,j}^n - \frac{\Delta t}{J_{i,j}\Delta\psi} \parenth{F^{n+1}_{i+\tfrac12,j}-F^{n+1}_{i-\tfrac12,j}} \\
				  &\quad - \frac{\Delta t}{J_{i,j}\Delta\theta} \parenth{G^{n+1}_{i,j+\tfrac12}-G^{n+1}_{i,j-\tfrac12}}.
\end{align*}

The jump conditions are imposed through \eqref{eqn:torus_BC_Y}, \eqref{eqn:torus_BC_v_pi}, and \eqref{eqn:torus_BC_v_0}, where the boundary flux values correspond to $\mathbf{Y}_q(\theta,\psi)=F_{i+\tfrac12,j}$.

The resulting nonlinear system at each time step is solved by a Newton--Krylov method, with Jacobian--vector products computed via automatic differentiation~\cite{kelley2003solving}.

For numerical simulation, the torus geometry is specified by the major and minor radii given by $R=2.0$ and $r=1.0$, respectively.
The drift vector field is chosen as
\begin{align}
    X_\psi(\theta,\psi)=3, \qquad
    X_\theta(\theta,\psi)=-6\sin(\theta-\psi),
\end{align}
and the diffusion coefficient is set to $H=0.5$.
The initial densities in the two modes are chosen to be the Gaussian in $\psi$ and uniform in $\theta$, namely,
\[
	v_1(0,\theta,\psi)=\mathcal{N}(\mu_1,\sigma_1^2)\cdot \mathbf{1}_\theta,
	\qquad
	v_2(0,\theta,\psi)=\mathcal{N}(\mu_2,\sigma_2^2)\cdot \mathbf{1}_\theta,
\]
where $\mathbf{1}_\theta$ denotes the uniform density on $\Sph^1$, and
\[
	\mu_1=\frac{3\pi}{4}, \quad \sigma_1=0.2, \quad \mu_2=\frac{5\pi}{4}, \quad \sigma_2=0.2.
\]
The grid resolution is set to $N_\psi=100$ and $N_\theta=100$, and the simulation is run up to $T=3$ using $N_t=1000$ time steps.

\begin{figure}
	\centering

	\subfloat[Density at $t=0 \tau$]{
		\scriptsize
		\begin{tikzpicture}
			\node at (0,0) {
			\includegraphics[
					width=0.46\columnwidth,
					trim={3cm 4cm 3cm 5cm},
					clip
				]{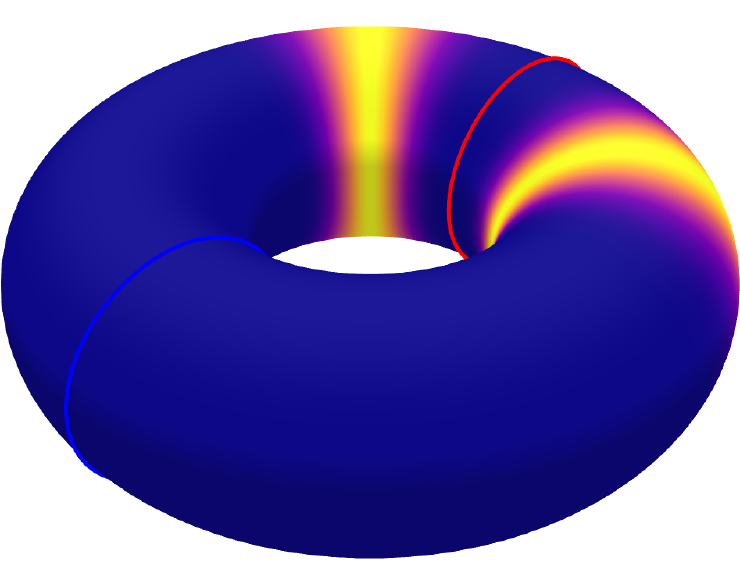}};
			\node at (1.6,0.7) {$\psi=\frac{3}{4}\pi$};
			\node at (0.0,1.6) {$\psi=\frac{5}{4}\pi$};
			\node at (-1.6, -1.3) {$\psi=0$};
			\node at (1.2,1.4) {$\G$};
		\end{tikzpicture}}\hfill
	\foreach \t [count=\i] in {2,4,8,12,18,24,48} {%
		\subfloat[Density at $t=\t \tau$]{
			\includegraphics[
			width=0.46\columnwidth,
			trim={3cm 4cm 3cm 5cm},
			clip
			]{./Figs/frames_torus/frame_\t.pdf}
		}
		\ifodd\i
			\\
		\else
			\hfill
		\fi
	}%

	\caption{Evolution of the density ($\tau = T/50 = 0.06$)}
	\label{fig:torus_torus}
\end{figure}

The resulting evolution of the density is illustrated in \Cref{fig:torus_torus},
where the red circle corresponds to the guard at $\psi=\pi$, and the blue circle represents the image of the reset at $\psi=0$.
In \Cref{fig:torus_torus}(a), the initial density is visualized as two rings centered at $\psi=\frac{3}{4}\pi$ and $\psi=\frac{5}{4}\pi$, respectively. 
The chosen drift field $X_\psi$ induces counterclockwise rotation, so the portion of the initial density at $\psi=\frac{3}{4}\pi$ is transported toward the guard (red circle) and subsequently reset to $\psi=0$ (blue circle) at the opposite phase of $\theta$, as shown in \Cref{fig:torus_torus}(b)--(c).
Meanwhile, the portion centered at $\psi=\frac{5}{4}\pi$ also rotates, but becomes less pronounced due to diffusion. 
After the jump, the density continues to rotate counterclockwise, while the drift component $X_\theta$ introduces a twisting effect, as illustrated in \Cref{fig:torus_torus}(d)--(f).
Over time, the combined effects of transport, reset, and diffusion lead to a steady-state behavior, forming a twisted band connecting the guard and the reset image.

\begin{figure}
    \centering
    \includegraphics[width=0.75\columnwidth]{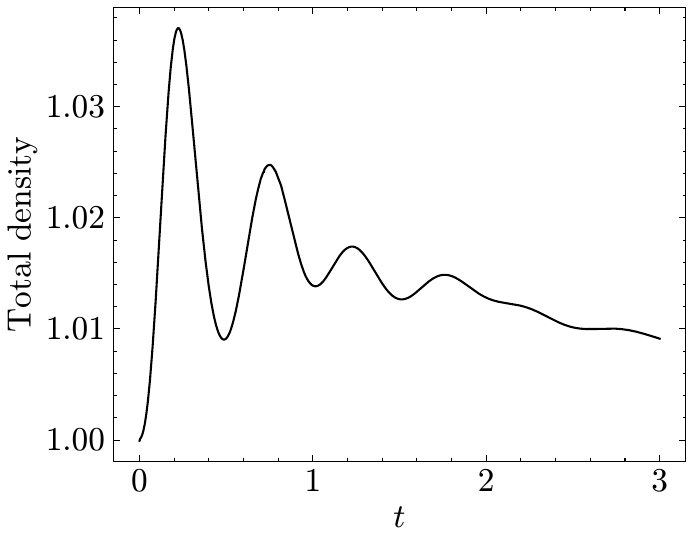}
	\caption{Evolution of the total probability mass}
    \label{fig:torus_totaldensity}
\end{figure}

To validate the numerical solution, we examine the total probability mass at each time step.
As shown in \Cref{fig:torus_totaldensity}, the proposed numerical scheme preserves the total mass with high accuracy.

Finally, we perform Monte Carlo simulations with a large ensemble of particles to approximate the density evolution.
The particle-based approximation, shown in \Cref{fig:comp_MC} on the $(\psi,\theta)$ plane, demonstrates excellent agreement with the finite-volume results.

This example demonstrates how the proposed formulation captures stochastic transport, diffusion, and reset within a unified framework on a manifold.
The finite-volume discretization preserves mass and accurately resolves the interaction between continuous flow and discrete jumps through the flux-based boundary conditions.
The agreement with Monte Carlo simulations further confirms that the weak-form treatment of the reset correctly transfers probability across the guard.
Overall, this example illustrates that the proposed approach provides a consistent and numerically tractable description of stochastic hybrid dynamics on a differentiable manifold.

\begin{figure}
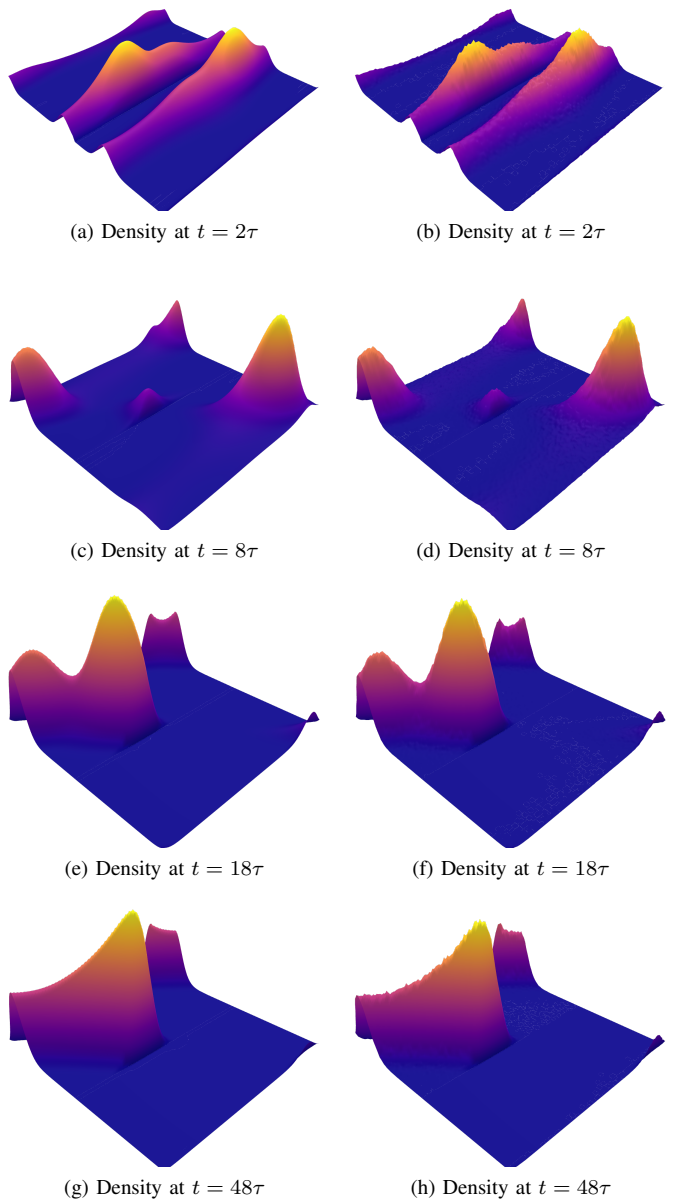

	\centering
	\foreach \n in {2, 8, 18, 48}{%
		\subfloat[Density at $t=\n \tau$]{\includegraphics[width=0.48\columnwidth, trim={1.5cm 0cm 1.5cm 6cm}, clip]{./Figs/frames_surface/frame_\n.pdf}} \hfill
	 	\subfloat[Density at $t=\n \tau$]{\includegraphics[width=0.48\columnwidth, trim={1.5cm 0cm 1.5cm 6cm}, clip]{./Figs/frames_MC_surface/frame_\n.pdf}}\\
	}%
	\caption{Verification of the computed density (left column) against Monte-Carlo simulation (right column)}\label{fig:comp_MC}
\end{figure}

\section{Conclusions}\label{sec:conclusions}

This paper presents a transfer operator framework for stochastic hybrid systems with guard-induced resets, encompassing both the Koopman and Frobenius--Perron operators.  
Exploiting their duality, we derived a unified description in which observables and probability densities evolve under adjoint generators corresponding to the backward and forward Kolmogorov equations.  
Focusing on uncertainty propagation, we developed a Frobenius--Perron formulation that couples within-mode Fokker--Planck dynamics with explicit inter-mode transfer terms.
The proposed framework is formulated intrinsically on differentiable manifolds, ensuring consistency with the underlying geometric structure.

Future work includes extensions to Bayesian estimation of hybrid systems, and scalable computational scheme for uncertainty propagation of higher-order hybrid systems. 

\bibliographystyle{IEEEtran}
\bibliography{ref} 

\appendix

\subsection{Divergence Theorem}\label{sec:divergence}

The divergence of a vector field can be obtained by the Lie derivative, i.e., $\div_\mu(X)\mu =\mathcal{L}_X \mu$. 
Also, from Cartan's formula, we have $\mathcal{L}_X \mu = d i_X \mu + i_X d\mu$, where $d\mu=0$ as it is a top form.  
By combining these, $\div_\mu(X)\mu = d i_X\mu$. 
According to Stoke's theorem, we obtain
\begin{align}
    \int_\X \div_\mu(X) \mu = \int_{\partial \X} \iota^* (i_X \mu),\label{eqn:div_thm0}
\end{align}
where $\iota: \partial\X\rightarrow\X$ denotes the inclusion map, and we take the pullback of $i_X\mu$ with the inclusion map to explicitly state that $i_X\mu$ is considered as a differential form on the boundary $\partial \X$. 
Also, the integration on the right hand side of \eqref{eqn:div_thm} is taken using the induced orientation on $\partial \X$. 

More explicitly, let $N:\partial \X \rightarrow \T \partial \X$ be the outward-pointing, unit normal vector field on $\partial \X$, and let $(v_1, \ldots, v_{n-1})\in (\T_x \partial \X)^{n-1}$ be a basis of $\T_x\partial\X$ for $x\in\partial \X$. 
Then, $(v_1, \ldots, v_{n-1})$ is positively oriented in $\T_x\partial \X$, if $(N, v_1,\ldots, v_{n-1})$ is positively oriented in $\T_x\X$, or $\mu(N, v_1,\ldots, v_{n-1}) > 0$. 

Further, we can decompose the vector field $X$ into $X=X_{\parallel} + X_{\perp}$, 
where $X_\parallel$ and $X_\perp$ denote the projection of $\X$ onto $\T_x\partial\X$ and the projection along $N$, respectively, i.e., 
\begin{align*}
    X_\perp = (X\cdot N) N,\quad X_\parallel = X- X_\perp.
\end{align*}
Then, using the linearity of a differential form, 
\begin{align*}
    \iota^* (i_X\mu) (v_1, \ldots v_{n-1}) & = \iota^* (i_{X_\perp} \mu) (v_1, \ldots v_{n-1}) \\
                                           & \quad + \iota^* (i_{X_\parallel} \mu) (v_1, \ldots v_{n-1}),
\end{align*}
where the last term vanishes as $X_\parallel$ can be written as a linear combination of $(v_1, \ldots, v_{n-1})$. 
These imply
\begin{align*}
    \iota^* (i_X\mu) = \iota^* (i_{X_\perp} \mu) = (X\cdot N) \; \iota^* (i_N \mu).
\end{align*}
Let $\nu = \iota^* (i_N\mu)$ be a differential form on $\partial \X$.
Finally, \eqref{eqn:div_thm0} is rewritten into
\begin{align}
    \int_\X \div_\mu(X) \mu = \int_{\partial \X} (X\cdot N) \nu.\label{eqn:div_thm}
\end{align}

\subsection{Integration by Parts}

For $f:\X\to\Re$, the product rule of divergence is given by
\begin{align}
    \div_\mu (fX) = f \div_\mu (X) + \mathcal{L}_X f.\label{eqn:div_prod}
\end{align}
Using \eqref{eqn:div_thm},
\begin{align}
    \int_\X \div_\mu (fX) \mu = \int_{\partial \X} f(X\cdot N) \nu. \label{eqn:div_tmp}
\end{align}
Combining \eqref{eqn:div_prod} and \eqref{eqn:div_tmp}, we obtain the following equation for integration by parts,
\begin{align}
    \int_\X f \div_\mu(X)\, \mu  = \int_{\partial \X} f(X\cdot N)\, \nu  - \int_\X \mathcal{L}_X f\, \mu. \label{eqn:integ_part}
\end{align}

\begin{IEEEbiography}
    {Tejaswi K. C.} received his Ph.D. in Mechanical and Aerospace Engineering at The George Washington University in 2024.
    He is continuing as a postdoctoral associate in Washington, D.C., U.S.A.
    His recent research consists of theoretical and computational methods in dynamics, optimization, data-driven methods and bio-inspired systems. 
\end{IEEEbiography}

\begin{IEEEbiography}
	{William A. Clark}
	William A. Clark received his B.S. degree in mechanical engineering from Ohio University in 2015 and his Ph.D. in a applied and interdisciplinary mathematics from the University of Michigan in 2020. 
	Between 2020 and 2023, he was a Post doctorate Research Fellow in the Department of Mathematics at Cornell University. 
	He is currently an Assistant Professor of Mathematics at Ohio University. His research interests include geometric mechanics, control of nonlinear systems, analysis of discontinuous systems, and geometric structures in learning.
\end{IEEEbiography}
\vskip 0pt plus -1fil
\begin{IEEEbiography}
    {Taeyoung Lee}
    		(Senior Member, IEEE) received the Ph.D. degree in aerospace engineering and the M.S. degree in mathematics from the University of Michigan, Ann Arbor, MI, USA, in 2008.
		He is a professor of the Department of Mechanical and Aerospace Engineering at the George Washington University, Washington, D.C., USA.
		His research interests include geometric mechanics and control with applications to complex aerospace systems.
\end{IEEEbiography}

\end{document}